\documentclass[12pt]{article}
\usepackage{fullpage}
\usepackage{epsfig}
\usepackage{enumerate}
\usepackage{enumitem}
\usepackage{xcolor}
\usepackage{url}

\usepackage{amssymb,amsmath,amsfonts,amsthm}

\newtheorem{theorem}{Theorem}[section]
\newtheorem{lemma}[theorem]{Lemma}

\newtheorem{observation}[theorem]{Observation}

\newtheorem{corollary}[theorem]{Corollary}

\def\df#1{{\em #1\/}}

\def\A{\mathcal{A}}
\def\B{\mathcal{B}}
\def\C{\mathcal{C}}

\def\E{\mathcal{E}}

\def\G{\mathcal{G}}
\def\H{\mathcal{H}}

\def\M{\mathcal{M}}

\def\P{\mathcal{P}}

\def\S{\mathcal{S}}

\def\NN{\mathbb{N}}
\def\RR{\mathbb{R}}
\def\SS{\mathbb{S}}

\newcounter{cases}
\newcounter{subcases}[cases]

\newenvironment{casesblock}{%
\setcounter{cases}{0}\setcounter{subcases}{0}\par}{}

\long\def\case#1{\stepcounter{cases}\medskip\noindent{\bf Case \arabic{cases}: }#1. \par}

\def\ifempty#1#2#3{%
\def\ifemptytest{#1}%
\def\ifemptyempty{}%
\ifx\ifemptyempty\ifemptytest #2\else #3\fi%
}

\def\iso{\cong}

\def\ss{\subseteq}

\def\sm{\setminus}

\def\into{\to}

\def\eg{\widehat{g}}
\def\ng{\widetilde{g}}

\def\+{{}^+\!}
\def\eg{\widehat{g}}
\def\egp{\widehat{g}\+}
\def\ng{\widetilde{g}}
\def\ngp{\widetilde{g}\+}
\def\Forb{\text{Forb}}
\def\dc#1#2{\Delta_{#1}(#2)}
\def\Gxy{\G_{xy}}
\def\Gcxy{\G^\circ_{xy}}
\def\Cc{\C^\circ}
\def\hat{\widehat}
\def\os{\sigma}
\def\osp{\sigma\+}
\def\eh{\widehat{h}}
\def\nh{\widetilde{h}}
\def\nt{\widetilde{\theta}}
\def\Hw{\H^{\rm w}}

\def\et{\theta}
\def\eeta{\eta}
\def\dd{d^*}
\def\homeo{\hookrightarrow}
\def\gs#1{(\ref{fg-separation}#1)}
\def\gsel#1{(\ref{fg-selected}#1)}
\def\gkur#1{(\ref{fg-cc-0-egp}#1)}
\def\gcas#1{(\ref{fg-cascades-2}#1)}
\def\gbase#1{(\ref{fg-bases}#1)}
\def\By{B_y}
\def\Bx{B_x}
\def\hem{\eta}

\begin{document}
\title{Cascades and Obstructions of Low Connectivity for Embedding Graphs into the Klein Bottle}
\author{%
  Bojan Mohar
  \thanks{Supported in part by an NSERC Discovery Grant (Canada),
    by the Canada Research Chair program, and by the
    Research Grant P1--0297 of ARRS (Slovenia).}~\thanks{On leave from:
    IMFM \& FMF, Department of Mathematics, University of Ljubljana, Ljubljana,
    Slovenia.}
  \qquad
  Petr \v{S}koda\\\\
  \normalsize
  Department of Mathematics, \\
  Simon Fraser University,\\
  8888 University Drive,\\
  Burnaby, BC, Canada.}
\date{}
\maketitle

\begin{abstract}
  The structure of graphs with a 2-vertex-cut that are critical with respect to the Euler genus is studied.
  A general theorem describing the building blocks is presented. These constituents, called hoppers and cascades, are classified for the case when Euler genus is small.
  As a consequence, the complete list of obstructions of connectivity 2 for embedding graphs into the Klein bottle is obtained.
\end{abstract}

\section{Introduction}

Robertson and Seymour~\cite{robertson-1990} proved that for each surface $\SS$
the class of graphs that embed into $\SS$ can be characterized by a finite list $\Forb(\SS)$ of minimal forbidden minors (or \df{obstructions}).
For the 2-sphere $\SS_0$, $\Forb(\SS_0)$ consists of the \df{Kuratowski graphs}, $K_5$ and $K_{3,3}$ \cite{kuratowski-1930}. The list of obstructions $\Forb(\NN_1)$
for the projective plane $\NN_1$ already contains 35 graphs and $\NN_1$ is the only other surface for which the complete list of forbidden minors is known \cite{archdeacon-1981,glover-1979}.
For the torus $\SS_1$, the complete list of obstructions is still not known, but thousands of obstructions
were generated by the use of computers (see~\cite{Chambers02,NeufeldMyrwold,Wood07}).

The obstructions for the Klein bottle are even less understood than those for the torus. Even though no list of obstructions have been constructed so far,
it is expected that the total number of obstructions for the Klein bottle will be in tens of thousand. Henry Glover (private communication to B.M.)
conjectured that there will be many more.
In fact, Glover made a speculation that more than $10^6$ obstructions will be obtained by pasting together two obstructions for the projective plane by identifying two vertices in all possible ways. One of the side results of this paper is a refutation of this conjecture.

In this paper, we study critical graphs for Euler genus of low connectivity.
For a graph $G$, we denote by $\eg(G)$ its Euler genus; see Section \ref{sc-preliminaries} for definitions.
A graph $G$ is \df{critical} for Euler genus $k$ if $\eg(G) > k$ and for each edge $e \in E(G)$, $\eg(G-e) \le k $ and $\eg(G / e) \le k$, where $G/e$ denotes the graph obtained from $G$ by contracting the edge $e$.
Let $\E_k$ be the class of critical graphs for Euler genus $k$ and $\E = \bigcup_{k\ge 0} \E_k$.
It is easy to show that graphs in $\E$ that are not 2-connected can be obtained as disjoint unions and 1-sums of graphs in $\E$ (see~\cite{stahl-1977}).
Here we study graphs in $\E$ of \df{connectivity} 2, that is, graphs that are 2-connected but not 3-connected.
We shall show that each critical graph for Euler genus of connectivity 2 can be obtained as a 2-sum of two graphs that are close to graphs in $\E$
or belong to an exceptional class of graphs, called \df{cascades} (see Sect.~\ref{sc-s1} and~\ref{sc-ext}).
In Sect.~\ref{sc-klein-euler}, we construct the list of critical graphs for Euler genus 2 of connectivity 2.
In Sect.~\ref{sc-klein-klein}, we show that a graph of connectivity 2 is critical for Euler genus 2 if and only if it is an obstruction for the Klein bottle.
This yields a complete list of obstructions for the Klein bottle of connectivity 2. The list of obstructions for embeddability in the Klein bottle (and for Euler genus 2) contains precisely 668 graphs of connectivity two.
This is in strong contrast with predictions of Henry Glover, who estimated that the number of Klein bottle obstructions of connectivity two will be more than a million (private communication).
An analogous result for the torus is given in~\cite{mohar-torus}. However, the methods used in that paper are quite different from those in this one. The main difference is the appearance of cascades, whose treatment occupies about half of this paper.

The above-mentioned result that obstructions of connectivity two for Euler genus 2 and the nonorientable genus 2 are the same is just a coincidence. It is easy to see that it no longer holds for larger genus. Also, there are 3-connected obstructions for Euler genus 2 that are not Klein bottle obstructions. One example is the following graph. Let $Q$ be the graph obtained from $K_7$ by first subdividing two of its edges that have no vertex in common and then adding an edge joining both vertices of degree two used in the subdivision. Since $K_7$ does not embed in the Klein bottle, $Q$ is not an obstruction for this surface. However, $Q$ cannot be embedded in the torus and, as the reader may verify, deleting or contracting any edge gives a graph of genus one. So, $Q$ is an obstruction for the torus and an obstruction for Euler genus 2.

In classifying obstructions of connectivity two, we encounter two special families of graphs that are the building blocks of such obstructions. The first class are mysterious graphs called {\em hoppers}. While we prove that hoppers do not exist when the genus is small, and we are not able to construct any for larger genus, we believe that they may show up when the genus is large enough. Their existence is closely related to an old open problem dating back to the 1980's asking if there exists a graph which is simultaneously an obstruction for two different nonorientable surfaces.\footnote{This problem was proposed in various incarnations by Dan Archdeacon, Bruce Richter, and Jozef \v{S}iran, and appears as Problem \#1 in the list of open problems in topological graph theory compiled by Dan Archdeacon in 1995 (\url{http://www.emba.uvm.edu/~darchdea/problems/decgenus.htm}).}
For such an obstruction, deleting or contracting any edge would reduce the nonorientable genus by at least two.

The graphs in the second family that we encounter are called {\em cascades}. We determine all cascades when the genus is small. The proofs use methods from structural graph theory and involve development of results about extensions of embeddings of subgraphs. The classification of cascades for Euler genus 2 occupies almost half of the paper and is the most complicated part of the paper.

In the first part of the paper, obstructions of connectivity two for arbitrary Euler genus are examined. It is shown that we encounter the same behavior as for the small genus, except that we are unable to say much about hoppers and cascades.

%------------------------------------ Preliminaries ----------------------------------------------------------
\section{Preliminaries}
\label{sc-preliminaries}

Let $G$ be  a connected multigraph. An \df{embedding} of $G$ is a pair $\Pi = (\pi, \lambda)$
where $\pi = (\pi_v \mid v \in V(G))$ is a \df{rotation system}, which assigns each vertex
$v$ a cyclic permutation $\pi_v$ of the edges incident with $v$, and $\lambda$ is a \df{signature}
mapping which assigns each edge $e \in E(G)$ a \df{sign} $\lambda(e) \in \{-1, 1\}$.
For an edge $e$ incident to $v$, the cyclic sequence $e, \pi_v(e), \pi_v^2(e), \ldots, e$ is called
the \df{local rotation} at $v$.
Given an embedding $\Pi$ of $G$, we say that $G$ is \df{$\Pi$-embedded}.

A \df{$\Pi$-face} of a $\Pi$-embedded graph $G$ is a cyclic sequence of triples $(v_i, e_i, s_i)$,
where $v_i \in V(G)$, $e_i$ is an edge incident with $v_i$, and $s_i\in \{1,-1\}$,
satisfying the following (with indices being cyclic):
\begin{enumerate}[label=(\roman*)]
\item
  $e_i = v_iv_{i+1}$,
\item
  $s_{i+1} = s_i\lambda(e_i)$, and
\item
  $e_{i+1} = \pi_{v_{i+1}}^{s_{i+1}}(e_i)$.
\end{enumerate}
Two consecutive tuples $(v, e, s), (v', e', s')$ of a $\Pi$-face $W$ give a \df{$\Pi$-angle} $e, v', e'$ of $W$.
Let $F(\Pi)$ be the set of $\Pi$-faces.
The \df{Euler genus} of $\Pi$ is given by Euler's formula.
$$|V(G)| - |E(G)| + |F(G)| = 2 - \eg(\Pi).$$
The \df{Euler genus} $\eg(G)$ of a graph $G$ is the minimum Euler genus of an embedding of $G$.

%% A \df{combinatorial embedding} $\Pi$ of a graph $G$ is a mapping that assigns, to each vertex $v \in V(G)$,
%% a permutation of edges incident with $v$ and, to each edge $e \in E(G)$, a \df{signature} $\Pi(e) \in \{-1, 1\}$.
%% A graph $G$ with a combinatorial embedding $\Pi$ can be seen as embedded in a surface $\SS$ (see~\cite{mohar-book}).
%% If $G$ contains a cycle that contains odd number of edges of negative signature,
%% we say that $\Pi$ is \df{nonorientable}. Otherwise, $\Pi$ is \df{orientable}.
%% A \df{facial walk} of $\Pi$ is a walk in $G$ that traces the boundary of a single region in the embedding of $G$ in $\SS$.
%% Let $F(\Pi)$ be the set of facial walks of $\Pi$.
%% The \df{genus} $g(\Pi)$ of $\Pi$ is given by the Euler's formula:
%% $$|V(G)| - |E(G)| + |F(\Pi)| = 2 - g(\Pi).$$

If $G$ contains a cycle that contains odd number of edges of negative signature,
we say that $\Pi$ is \df{nonorientable}. Otherwise, $\Pi$ is \df{orientable}.
The \df{orientable} genus $g(G)$ is half of the minimum genus of an orientable combinatorial embedding of $G$.
If $G$ contains at least one cycle, then the \df{nonorientable genus} $\ng(G)$ is the minimum
Euler genus of a nonorientable embedding of $G$, else $\ng(G) = 0$.
The following relation is an easy observation (see~\cite{mohar-book}).

\begin{lemma}
\label{lm-ng-upper}
  For every connected graph $G$ which is not a tree,
  $$\ng(G) \le 2g(G) + 1.$$
\end{lemma}

If $\ng(G) = 2g(G) + 1$, then $G$ is said to be \df{orientably simple}.
Note that in this case $\eg(G) = 2g(G) = \ng(G)-1$, i.e., the Euler genus of $G$ is even.

In this paper, we will deal mainly with the class $\G$ of simple graphs.
Let $G \in \G$ be a simple graph and $e$ an edge of $G$.
Then $G-e$ denotes the graph obtained from $G$ by \df{deleting} $e$
and $G/e$ denotes the graph\footnote{When contracting an edge, one may obtain multiple edges.
We shall replace any multiple edges by single edges as such a simplification has no effect on the genus.}
obtained from $G$ by \df{contracting} $e$. It is convenient for us to formalize these graph operations.
The set $\M(G) = E(G) \times \{-,/\}$ is the \df{set of minor-operations} available for $G$.
An element $\mu \in \M(G)$ is called a \df{minor-operation} and $\mu G $ denotes the graph obtained
from $G$ by applying $\mu$. For example, if $\mu = (e, -)$ then $\mu G  = G - e$.
A graph $H$ is a \df{minor} of $G$ if $H$ can be obtained from a subgraph of $G$ by contracting some edges.
If $G$ is connected, then $H$ can be obtained from $G$ by a sequence of minor-operations.

We shall use the following well-known result.

\begin{theorem}[Stahl and Beineke~\cite{stahl-1977}]
\label{th-stahl-euler}
The Euler genus of a graph is the sum of the Euler genera of its blocks.
\end{theorem}

Generally, we are interested in minor-minimal graphs (with some property). The closely related classes
of deletion-minimal graphs appear naturally.
Let $\Forb^*(\SS)$ be the class of graphs of minimum degree at least 3 that do not embed into $\SS$ but are minimal such with respect to taking subgraphs.
Similarly, let $\E^*_k$ be the class of graphs of minimum degree at least 3 such that $\eg(G) > k$ but $\eg(G - e) \le k$ for each edge $e \in E(G)$.
Again, we let $\E^* = \bigcup_{k\ge 0} \E^*_k$.

Let us note that the neighbors of a vertex of degree 3 cannot be adjacent in a graph that is minimally non-embeddable on a surface.

\begin{observation}
\label{obs-triangle}
  Let $uvw$ be a triangle in a graph $G$. If $u$ has degree 3, then
  every embedding of $G - vw$ into a surface can be extended to an embedding of $G$ into the same surface.
\end{observation}

%-------------------------------------- Graphs with terminals -----------------------------------------------------------
\section{Graphs with terminals}
\label{sc-terminals}

We study the class $\Gxy$ of graphs with two special vertices $x$ and $y$, called \df{terminals}.
Most notions that are used for graphs can be used in the same way for graphs with terminals.
Some notions differ though and, to distinguish between graphs with and without terminals, let $\hat{G}$ be the underlying graph of $G$ without terminals (for $G \in \Gxy$).
Two graphs, $G_1$ and $G_2$, in $\Gxy$ are \df{isomorphic}, also denoted $G_1 \iso G_2$, if there is an isomorphism of the graphs $\hat{G_1}$ and $\hat{G_2}$
that maps terminals of $G_1$ onto terminals of $G_2$ (and non-terminals onto non-terminals), possibly exchanging $x$ and $y$.
We define minor-operations on graphs in $\Gxy$ in the way that $\Gxy$ is a minor-closed class.
When performing edge contractions on $G \in \Gxy$, we do not allow contraction of the edge $xy$ (if $xy \in E(G)$)
and when contracting an edge incident with a terminal, the resulting vertex becomes a terminal.

We use $\M(G)$ to denote the set of available minor-operations for $G$.
Since $(xy, /) \not\in \M(G)$ for $G \in \Gxy$,
we shall use $G /xy$ to denote the underlying simple graph in $\G$ obtained from $G$
by identification of $x$ and $y$. In particular, we do not require the edge $xy$ to be present in~$G$.

A \df{graph parameter} is a function $\G \into \RR$ that is constant on each isomorphism class of $\G$.
Similarly, we call a function $\Gxy \into \RR$ a \df{graph parameter} if it is constant on each isomorphism class of $\Gxy$.
A graph parameter $\P$ is \df{minor-monotone} if $\P(H) \le \P(G)$ for each graph $G \in \Gxy$ and each minor $H$ of $G$.
The Euler genus is an example of a minor-monotone graph parameter.

For $G \in \Gxy$, the graph $G\+$ is the graph obtained from $G$ by adding the edge $xy$ if it is not already present.
We can view the Euler genus of $G\+$ as a graph parameter $\egp$ of $G$, $\egp(G) = \eg(G\+)$.
Note that $\egp$ is minor-monotone.
The difference of $\egp$ and $\eg$ is a parameter $\et$, that is, $\et(G) = \eg(G\+) - \eg(G)$.
Note that $\et(G) \in \{0,1,2\}$.

\begin{table}
  \centering
  \begin{tabular}{| c | l | l |}
    \hline
    Parameter & Definition & Range \\
    \hline
    $\eg(G)$ & Euler genus & $\ge 0$ \\
    $\egp(G)$ & $\eg(G+xy)$ & $\ge 0$ \\
    $\et(G)$ & $\egp(G)-\eg(G)$ & $0,1,2$ \\
    $\eeta(G_1, G_2)$ & $\et(G_1) + \et(G_2)$ & $0,1,2,3,4$ \\
    \hline
  \end{tabular}
  \caption{Genus parameters for graphs in $\Gxy$}
  \label{tb-ranges}
\end{table}

Let $\P$ be a graph parameter.
A graph $G$ is \df{$\P$-critical} if $\P(\mu G) < \P(G)$ for each $\mu \in \M(G)$.
Let $H$ be a subgraph of a graph $G$ (possibly with terminals) and $\P$ a graph parameter. We say that $H$ is \df{$\P$-tight} if
$\P(\mu G) < \P(G)$ for every minor-operation $\mu \in \M(H)$.
We observe that $\P$-critical graphs have $\P$-tight subgraphs:

\begin{lemma}
\label{lm-tight}
Let $H_1, \ldots, H_s$ be subgraphs of a graph $G$ (possibly with terminals).
If $E(H_1) \cup \cdots \cup E(H_s) = E(G)$, then
$G$ is $\P$-critical if and only if $H_1, \ldots, H_s$ are $\P$-tight.
\end{lemma}

Let $\Gcxy$ be the subclass of $\Gxy$ that consists of graphs that do not contain the edge $xy$.
For graphs $G_1, G_2 \in \Gxy$ such that $V(G_1) \cap V(G_2) = \{x,y\}$, the graph $G = (V(G_1) \cup V(G_2), E(G_1) \cup E(G_2))$ is the \df{$xy$-sum} of $G_1$ and $G_2$.
The graphs $G_1$ and $G_2$ are called \df{parts} of $G$.
Let $G$ be the $xy$-sum of $G_1, G_2 \in \Gxy$.
We define the following two parameters:
\begin{equation}
  \label{eq-eg-0}
  \eh_0(G) = \eg(G_1) + \eg(G_2) + 2;
\end{equation}
\begin{equation}
  \label{eq-eg-1}
  \eh_1(G) = \egp(G_1) + \egp(G_2).
\end{equation}

Eq.~(\ref{eq-eg-1}) can be rewritten in a form similar to Eq.~(\ref{eq-eg-0}).
\begin{equation}
  \label{eq-thetas}
  \eh_1(G) = \eg(G_1) + \eg(G_2) + \et(G_1) + \et(G_2) = \eg(G_1) + \eg(G_2) + \eeta(G_1, G_2)
\end{equation}
where $\eeta(G_1, G_2) = \et(G_1) + \et(G_2)$.
Note that $\eeta(G_1, G_2) \in \{0,1,2,3,4\}$.

Richter~\cite{richter-1987-euler} gave a precise formula for the Euler genus of a 2-sum that
can be expressed using our notation as follows.

\begin{theorem}[Richter~\cite{richter-1987-euler}]
  \label{th-richter-euler}
  Let $G$ be the $xy$-sum of connected graphs $G_1, G_2 \in \Gxy$. Then
  \begin{enumerate}[label=\rm(\roman*)]
  \item
    $\eg(G) = \min \{\eh_0(G), \eh_1(G)\}$,
  \item
    $\egp(G) = \eh_1(G)$, and
  \item
    $\et(G) = \max\{\eh_1(G) - \eh_0(G), 0\}.$
  \end{enumerate}
\end{theorem}

We can rewrite (i) as
\begin{equation}
  \label{eq-eta-eg}
  \eg(G) = \eg(G_1) + \eg(G_2) + \min\{\eeta(G_1, G_2), 2\}
\end{equation}
and as
\begin{equation}
  \label{eq-eta-egp}
  \eg(G) = \egp(G_1) + \egp(G_2) + 2 - \max\{\eeta(G_1, G_2), 2\}.
\end{equation}

For a graph parameter $\P$, we say that a minor-operation $\mu \in \M(G)$ \df{decreases} $\P$ by at least $k$
if $\P(\mu G) \le \P(G) - k$. The subset of $\M(G)$ that decreases $\P$ by at least $k$ is denoted by $\dc k {\P,G}$.
We write just $\dc k \P$  when the graph is clear from the context.
Note that a graph $G$ is $\P$-critical precisely when $\M(G) = \dc1 \P$.
The following observation is stated for later reference.

\begin{lemma}
\label{lm-two-separated}
  Let $G \in \Gxy$. Then $\eg(G) \le \egp(G) \le \eg(G) + 2$. Furthermore, for $\et = \et(G)$ and $k \ge 0$, we have
  \begin{enumerate}
  \setlength{\itemindent}{3mm}
  \item[\rm(S1)]
    $\dc {k+2-\et} \eg \ss \dc k \egp$ and
  \item[\rm(S2)]
    $\dc {k+\et} \egp \ss \dc k \eg$.
  \end{enumerate}
\end{lemma}

\begin{proof}
Suppose that $\mu \in \dc {k+2-\et} \eg$, i.e., $\eg(G)\ge \eg(\mu G)+k+2-\et$.
Then $\egp(G) = \eg(G)+\et \ge \eg(\mu G)+k+2 \ge \egp(\mu G) + k$. This shows that $\mu\in \dc {k} \egp$ and proves (S1). Property (S2) is verified in the same way.
\end{proof}

As an example, take a graph $G$ with $\eg(G) = 1$ and $\egp(G) = 2$. Then (S2) for $k = 1$ says that
$\dc2 \egp \ss \dc1 \eg$, or that each minor-operation that decreases the Euler genus of $G\+$ by at least 2 also decreases
the Euler genus of $G$ by at least $1$.

The next lemma describes when a minor-operation in a part of a 2-sum decreases $\eg$ of the 2-sum.

\begin{lemma}
  \label{lm-part-eg}
  Let $G$ be the $xy$-sum of connected graphs $G_1, G_2 \in \Gcxy$ and let $\mu \in \M(G_1)$ be a minor-operation such that $\mu G_1$ is connected.
  Then $\eg(\mu G) < \eg(G)$ if and only if the following is true (where $\dc k\cdot$ always refer to the decrease of the parameter in $G_1$):
  \begin{enumerate}[label=\rm(\roman*)]
  \item
    If $\eeta(G_1, G_2) = 0$, then $\mu \in \dc1 \egp$.
  \item
    If $\eeta(G_1, G_2) = 1$, then $\mu \in \dc1 \egp \cup \dc2 \eg$.
  \item
    If $\eeta(G_1, G_2) = 2$, then $\mu \in \dc1 \egp \cup \dc1 \eg$.
  \item
    If $\eeta(G_1, G_2) = 3$, then $\mu \in \dc2 \egp \cup \dc1 \eg$.
  \item
    If $\eeta(G_1, G_2) = 4$, then $\mu \in \dc1 \eg$.
  \end{enumerate}
\end{lemma}

\begin{proof}
  Assume first that $\eg(\mu G) < \eg(G)$.
  Suppose that $\eeta(G_1, G_2) \le 2$ and that $\mu \not\in \dc1 \egp$.
  Since $\mu G_1$ is connected and $\eh_1(\mu G) = \eh_1(G)$, Theorem~\ref{th-richter-euler} gives that $\eg(\mu G) = \eh_0(G)$.
  Thus using Eq.~(\ref{eq-eta-eg}), we obtain that
  $$\eg(\mu G_1) + \eg(G_2) + 2 = \eh_0(\mu G) = \eg(\mu G) < \eg(G) \le \eg(G_1) + \eg(G_2) + \eeta(G_1, G_2).$$
  If $\eeta(G_1, G_2) = 0$, then $\eg(\mu G_1) < \eg(G_1) - 2$. Thus $\mu \in \dc3 \eg$ and, since  $\dc3 \eg \ss \dc1 \egp$ by (S1), $\mu \in \dc1 \egp$, a contradiction.
  We conclude that (i) holds.
  If $\eeta(G_1, G_2) = 1$, then $\mu \in \dc2 \eg$ and (ii) holds.
  If $\eeta(G_1, G_2) = 2$, then $\mu \in \dc1 \eg$ and (iii) holds.

  Assume now that $\eeta(G_1, G_2) \ge 3$ and assume that $\mu \not\in \dc1 \eg$. Then $\eh_0(\mu G) = \eh_0(G)$.
  Consequently, by Theorem~\ref{th-richter-euler}, $\eg(\mu G) = \eh_1(G)$. Thus using Eq.~(\ref{eq-eta-egp}), we have
  $$\egp(\mu G_1) + \egp(G_2) = \eh_1(\mu G) = \eg(\mu G) < \eg(G) \le \egp(G_1) + \egp(G_2) + 2 - \eeta(G_1, G_2).$$
  If $\eeta(G_1, G_2) = 3$, then $\egp(\mu G_1) < \egp(G_1) - 1$. Hence $\mu \in \dc2 \egp$ and (iv) holds.
  If $\eeta(G_1, G_2) = 4$, then $\egp(\mu G_1) < \egp(G_1) - 2$. Thus $\mu \in \dc3 \egp$ and, since $\dc3 \egp \ss \dc1 \eg$ by (S2), $\mu \in \dc1 \eg$, a contradiction.
  We conclude that (v) holds.

  To prove the ``if'' part of the lemma, assume that (i)--(v) hold.
  We need to show that $\eg(\mu G) < \eg(G)$.
  Assume first that $\eeta(G_1, G_2) \le 2$ and $\mu \in \dc1 \egp$. By Theorem~\ref{th-richter-euler} and~(\ref{eq-eta-egp}),
  $$\eg(\mu G) \le \eh_0(\mu G) = \egp(\mu G_1) + \egp(G_2) < \egp(G_1) + \egp(G_2) = \eh_0(G) = \eg(G).$$

  Assume now that $\eeta(G_1, G_2) \ge 2$ and $\mu \in \dc1 \eg$. By Theorem~\ref{th-richter-euler} and~(\ref{eq-eta-eg}),
  $$\eg(\mu G) \le \eh_1(\mu G) = \eg(\mu G_1) + \eg(G_2) + 2 < \eg(G_1) + \eg(G_2) + 2 = \eh_1(G) = \eg(G).$$

  If $\eeta(G_1, G_2) = 1$ and $\mu \in \dc2 \eg$, then using~(\ref{eq-eta-eg}),
  $$\eg(\mu G) \le \eh_1(\mu G) = \eg(\mu G_1) + \eg(G_2) + 2 < \eg(G_1) + \eg(G_2) + 1 = \eg(G).$$

  If $\eeta(G_1, G_2) = 3$ and $\mu \in \dc2 \egp$, then using~(\ref{eq-eta-egp}),
  $$\eg(\mu G) \le \eh_0(\mu G) = \egp(\mu G_1) + \egp(G_2) < \egp(G_1) + \egp(G_2) -1 = \eg(G).$$
  Since the cases (i)--(v) cover all possible values of $\eeta$, at least one of the cases above occurs and we are done.
\end{proof}

Let us prove a similar lemma for $\egp$.

\begin{lemma}
  \label{lm-part-egp}
  Let $G$ be the $xy$-sum of connected graphs $G_1, G_2 \in \Gcxy$ and let $\mu \in \M(G_1)$ be a minor-operation such that $\mu G_1$ is connected.
  Then $\egp(\mu G) < \egp(G)$ if and only if $\mu \in \dc1 \egp$.
\end{lemma}

\begin{proof}
  Assume first that $\egp(\mu G) < \egp(G)$.
  Since $\mu G_1$ is connected, Theorem~\ref{th-richter-euler} gives that $\egp(\mu G) = \eh_1(G)$. Using Eq.~(\ref{eq-eg-1}), we have that
  $$\egp(\mu G_1) + \egp(G_2) = \eh_1(\mu G) = \egp(\mu G) < \egp(G) = \eh_1(G) = \egp(G_1) + \eg(G_2).$$
  Thus $\egp(\mu G_1) < \egp(G_1)$. Hence $\mu \in \dc1 \egp$.

  On the other hand, assume that $\mu \in \dc1 \egp$.
  Thus $\egp(\mu G_1) < \egp(G_1)$.
  By Theorem~\ref{th-richter-euler} and Eq.~(\ref{eq-eg-1}),
  $$\egp(\mu G) = \egp(\mu G_1) + \egp(G_2) < \egp(G_1) + \egp(G_2) = \egp(G),$$
  as claimed.
\end{proof}

In the statements of Lemmas~\ref{lm-part-eg} and~\ref{lm-part-egp}, we required that $\mu G_1$ is connected.
The next lemma shows that this is indeed the case for all minor-operations if $G_1$ is $\eg$-tight or $\egp$-tight in $G$.
It is not hard to see that if $G_1$ is a connected graph, then $\mu G_1$ is disconnected if and only if $\mu$ is the deletion of a cutedge of $G_1$.

\begin{lemma}
  \label{lm-no-cutedges}
  Let $G \in \Gcxy$ be a connected graph with a cutedge $e$.
  Then $\eg(G/e) = \eg(G)$ and $\egp(G/e) = \egp(G)$.
\end{lemma}

\begin{proof}
  Let $H_1$ and $H_2$ be the components of $G - e$.
  By Theorem~\ref{th-stahl-euler}, $\eg(G/e) = \eg(H_1) + \eg(H_2) = \eg(G)$.
  If both $x$ and $y$ lie in $H_1$ (or $H_2$ by symmetry), then
  by Theorem~\ref{th-stahl-euler}, $\egp(G/e) = \eg(H_1 + xy) + \eg(H_2) = \egp(G)$.
  Suppose then that $x \in V(H_1)$ and $y \in V(H_2)$.
  If $x$ (or $y$ by symmetry) and $w$ are the endpoints of $e$, then $G\+$ is the $1$-sum of $H_1$ and $H_2 + xy + xw$.
  Since $\eg(H_2 + yw) = \eg(H_2 + xy + xw)$, we have that $\egp(G/e) = \eg(H_1) + \eg(H_2 + yw) = \eg(H_1) + \eg(H_2 + xy + e) = \egp(G)$.

  Therefore we may assume that $e$ has endpoints $z \in V(H_1) \sm \{x\}$ and $w \in V(H_2) \sm \{y\}$.
  Let us view the graph $G\+$ as a $yz$-sum of graphs $H_1' = H_1+xy$ and $H_2' = H_2 + e$.
  We have that $(e,/) \in \M(H_2')$ and $\eg(H_2'/e) = \eg(H_2')$ by Theorem~\ref{th-stahl-euler} since $e$ is a block of $H_2'$.
  Similarly, $\egp(H_2'/e) = \egp(H_2')$ since $H_2'/e$ is homeomorphic to $H_2'$ and thus admits the same embeddings.
  By applying Theorem~\ref{th-stahl-euler} to $G\+$ as a $yz$-sum of $H_1'$ and $H_2'$, we obtain that $\eg(G\+/e) = \eg(G\+)$.
  We conclude that $\egp(G/e) = \egp(G)$.
\end{proof}

\begin{table}
  \centering
  \begin{tabular}{|c | c |}
    \hline
    $\eeta(G_1, G_2)$ & $\M(G_1)$ \\
    \hline
    0 & $\dc1 \egp$ \\
    1 & $\dc1 \egp \cup \dc2 \eg$ \\
    2 & $\dc1 \egp \cup \dc1 \eg$ \\
    3 & $\dc2 \egp \cup \dc1 \eg$ \\
    4 & $\dc1 \eg$ \\
    \hline
  \end{tabular}
  \caption{Possible outcomes for a minor-operation in a $\eg$-tight part of a 2-sum.}
  \label{tb-parts-eg}
\end{table}

Lemma~\ref{lm-no-cutedges} easily implies that a $\eg$-tight or $\egp$-tight part $G_1$ of an $xy$-sum $G$
has no cutedges. If $e$ is a cutedge of $G_1$, then $G_1 / e$ is connected, $\eg(G_1 /e) = \eg(G_1)$, and $\egp(G_1/e) = \egp(G_1)$.
By Lemmas~\ref{lm-part-eg} and~\ref{lm-part-egp}, $G_1$ is neither $\eg$-tight nor $\egp$-tight in $G$.
In particular, we may present the outcome of Lemma~\ref{lm-part-eg} in terms of $\M(G_1)$ as in Table~\ref{tb-parts-eg}.

%---------------------------------- Critical classes, cascades, and hoppers --------------------------------------------
\section{Critical classes, cascades, and hoppers}
\label{sc-klein-critical}

For a graph parameter $\P$, let $\C(\P)$ denote the class of $\P$-critical graphs in $\Gxy$.
Note that $G \in \C(\P)$ if and only if $\M(G) = \dc1 \P$.
We call $\C(\P)$ the \df{critical class} for $\P$.
Let $\Cc(\P)$ be the class $\C(\P) \cap \Gcxy$.
We refine the class $\C(\P)$ according to the value of $\P$:
Let $\C_k(\P)$ denote the subclass of $\C(\P)$ that contains precisely the graphs $G$ for which $\P(G) = k+1$. Let $\Cc_k(\P)$ be the class $\C_k(\P) \cap \Gcxy$ of those $\P$-critical graphs that do not contain the edge $xy$.

Let us start this section by describing the relation between the classes $\Cc(\eg)$, $\Cc(\egp)$, and $\E$ (unlabeled graphs that are critical for the Euler genus).
The next result follows from the definitions of $\E$ and $\Cc(\eg)$.

\begin{lemma}
\label{lm-cc-eg}
  For $G \in \Gcxy$, $\hat{G} \in \E$ if and only if $G \in \Cc(\eg)$.
\end{lemma}

The next two lemmas describe the relation between the class $\Cc(\egp)$ and $\E$.

\begin{lemma}
\label{lm-cc-egp-iff}
  For $G \in \Gcxy$, $\hat{G\+} \in \E$ if and only if $G \in \Cc(\egp)$, $\et(G) > 0$, and $\eg(G / xy) < \egp(G)$.
\end{lemma}

\begin{proof}
  Let $H = \hat{G\+}$.
  Note that $\eg(H) = \egp(G)$ and $\M(H) = \M(G) \cup \{(xy, -), (xy, /)\}$.
  Since $\eg(\mu H) = \egp(\mu G)$ for each $\mu \in \M(G)$,
  we get that $\eg(\mu H) < \eg(H)$ for each $\mu \in \M(G)$ if and only if $G \in \Cc(\egp)$.
  Since $H - xy \iso \hat{G}$, we obtain that $\eg(H - xy) < \eg(H)$ if and only if $\et(G) > 0$.
  Since $H/xy \iso G/xy$, we have that $\eg(H / xy) < \eg(H)$ if and only if $\eg(G / xy) < \egp(G)$.
  As $H \in \E$ if and only if $\eg(\mu H) < \eg(H)$ for each $\mu \in \M(H)$, the result follows.
\end{proof}

\begin{lemma}
\label{lm-cc-egp}
Let $G \in \Cc(\egp)$. If $\et(G) = 0$, then $\hat{G} \in \E$.
If $\et(G) > 0$, then either $\hat{G\+} \in \E$, or $\hat{G\+} \in \E^*$ and $\hat{G / xy} \in \E$.
\end{lemma}

\begin{proof}
  If $\et(G) = 0$, then $\M(G) = \dc1 \eg$ by (S2) and thus $G \in \Cc(\eg)$.
  Therefore $\hat{G} \in \E$ by Lemma~\ref{lm-cc-eg}.
  Suppose now that $\et(G) > 0$. Let $H = \hat{G\+}$.
  Since $G \in \Cc(\egp)$, we have that $\eg(\mu H) < \eg(H)$ for each $\mu \in \M(G)$.
  As $\eg(H - xy) = \eg(G) < \eg(G) + \et(G) = \eg(H)$,
  we have that $H \in \E^*$.
  If $\eg(G / xy) < \egp(G)$, then $H \in \E$ (since both deletion and contraction of $xy$ decrease the Euler genus of $H$).
  Hence we may assume that $\eg(G /xy) = \egp(G)$.
  Let $\mu \in \M(G /xy)$ be a minor-operation in $G / xy$.
  Since $\mu$ is also a minor-operation in $G$, we obtain that
  $$\eg(\mu(G /xy)) \le \eg(\mu G\+) = \egp(\mu G) < \egp(G) = \eg(G / xy)$$
  as $\mu(G /xy)$ is a minor of $\hat{\mu G\+}$.
  Since $\mu$ was chosen arbitrarily, $G/xy \in \E$.
\end{proof}

A graph $G \in \Gcxy$ is called a \df{cascade} if $G$ satisfies the following properties:
\begin{enumerate}[label=(C\arabic*)]
\item
  $\M(G) = \dc1 \eg \cup \dc1 \egp$ (i.e., each minor operation decreases $\eg$ or $\egp$).
\item
  $G \not\in \Cc(\eg)$ (i.e., some minor operation does not decrease $\eg$).
\item
  $G \not\in \Cc(\egp)$ (i.e., some minor operation does not decrease $\egp$).
\end{enumerate}

Let $\S$ be the class of all cascades.
We refine the class $\S$ according to the Euler genus.
Let $\S_k$ be the subclass of $\S$ containing those graphs $G$ for which $\egp(G) = k+1$.
It is not hard to see that for $G \in \S_k$ we have that $\eg(G) = k$.

\begin{lemma}
  \label{lm-cascades}
  If $G \in \S$, then $\et(G) = 1$.
\end{lemma}

\begin{proof}
  If $\et(G) = 0$, then $\dc1 \egp \ss \dc1 \eg$ by (S2), violating (C2).
  If $\et(G) = 2$, then $\dc1 \eg \ss \dc1 \egp$ by (S1), violating (C3).
  Thus $\et(G) = 1$.
\end{proof}

In this paper we shall show that the class of cascades is nonempty.
In particular, we will determine the class $\S_1$ which appears as a class of building blocks for obstructions of connectivity 2 for the Klein bottle.
The following lemma is an immediate consequence of (C1)--(C3).

\begin{lemma}
\label{lm-eg-and-egp}
  Let $G \in \Gcxy$. If $\M(G) = \dc1 \eg \cup \dc1 \egp$, then
  $G \in \Cc(\eg) \cup \Cc(\egp) \cup \S$.
\end{lemma}

%------------------------- hoppers  ------------------------------------

We shall encounter another class of building blocks for obstructions of connectivity two. This class is more mysterious and we call them hoppers. Although it turns out that they do not exist when the genus is small (see Lemma~\ref{lm-hoppers}), we suspect that they might appear when the genus becomes large. Their existence or nonexistence is intimately related to an old open question if there exist graphs that are obstructions for two different nonorientable surfaces.

Let $G \in \Gcxy$.
For a graph parameter $\P$, a graph $G$ is a \df{$\P$-hopper} if every minor operation reduces the parameter by at least 2, i.e., $\M(G) = \dc2 \P$.
Let $\H(\P)$ be the class of $\P$-hoppers.
The subclass of $\H(\P)$ of graphs with $\P$ equal to $k+1$ is denoted by $\H_k(\P)$.
In this paper, we restrict our attention to $\eg$-hoppers and $\egp$-hoppers.

Let us define two weaker forms of hoppers.
We say that $G$ is a \df{weak $\eg$-hopper} if $G \not\in \Cc(\egp)$ and $\M(G) = \dc1 \egp \cup \dc2 \eg$.
Note that necessarily $\et(G) = 0$ by (S1); and $G \in \Cc(\eg)$ by (S2).
We say that $G$ is a \df{weak $\egp$-hopper} if $G \not\in \Cc(\eg)$ and $\M(G) = \dc1 \eg \cup \dc2 \egp$.
Note that $\et(G) = 2$ by (S2) and $G \in \Cc(\egp)$ by (S1).
Let $\Hw(\eg)$ and $\Hw(\egp)$ be the class of weak $\eg$-hoppers and weak $\egp$-hoppers, respectively.
Let $\Hw_k(\P)$ be the subclass of $\Hw(\P)$ such that $G \in \Hw_k(\P)$ if $\P(G) = k+1$.
The next result follows directly from the definition of weak hoppers.

\begin{lemma}
\label{lm-weak-hoppers}
  Let $G \in \Gcxy$.
  If $\M(G) = \dc2 \eg \cup \dc1 \egp$, then $G \in \Cc(\egp) \cup \Hw(\eg)$.
  If $\M(G) = \dc1 \eg \cup \dc2 \egp$, then $G \in \Cc(\eg) \cup \Hw(\egp)$.
\end{lemma}

For the record we also state the following observation.

\begin{observation}
  The class $\H_k(\eg)$ is empty if and only if each graph $G \in \E_{k-1}$ has $\eg(G) = k$.
\end{observation}

%---------------------- general theorem ------------------------------------
% \section{Main theorem}
% \label{sc-main}

Let us now combine the properties of introduced classes with Lemma~\ref{lm-part-eg} to
characterize $\eg$-tight and $\egp$-tight parts of a 2-sum of two graphs.

\begin{table}
  \centering
  \begin{tabular}{|c | c |}
    \hline
    $\eeta(G_1, G_2)$ & $G_1$ \\
    \hline
    0 & $\Cc(\egp)$ \\
    1 & $\Cc(\egp) \cup \Hw(\eg)$ \\
    2 & $\Cc(\egp) \cup \Cc(\eg) \cup \S$ \\
    3 & $\Cc(\eg) \cup \Hw(\egp)$ \\
    4 & $\Cc(\eg)$ \\
    \hline
  \end{tabular}
  \caption{Classification of $\eg$-tight parts of a 2-sum.}
  \label{tb-general-eg}
\end{table}

\begin{theorem}
  \label{th-general-eg}
  Let $G$ be the $xy$-sum of connected graphs $G_1, G_2 \in \Gcxy$.
  The subgraph $G_1$ is $\eg$-tight in $G$ if and only if the following is true:
  \begin{enumerate}[label=\rm(\roman*)]
  \item
    If $\eeta(G_1, G_2) = 0$, then $G_1 \in \Cc(\egp)$.
  \item
    If $\eeta(G_1, G_2) = 1$, then $G_1 \in \Cc(\egp) \cup \Hw(\eg)$.
  \item
    If $\eeta(G_1, G_2) = 2$, then $G_1 \in \Cc(\egp) \cup \Cc(\eg) \cup \S$.
  \item
    If $\eeta(G_1, G_2) = 3$, then $G_1 \in \Cc(\eg) \cup \Hw(\egp)$.
  \item
    If $\eeta(G_1, G_2) = 4$, then $G_1 \in \Cc(\eg)$.
  \end{enumerate}
\end{theorem}

\begin{proof}
  Assume first that $G_1$ is $\eg$-tight.
  By Lemma~\ref{lm-no-cutedges}, $\mu G_1$ is connected for each $\mu \in \M(G_1)$.
  If $\eeta(G_1, G_2) = 0$, then $\M(G_1) = \dc1 \egp$ by Lemma~\ref{lm-part-eg}.
  Thus $G_1 \in \Cc(\egp)$. Similarly, if $\eeta(G_1, G_2) = 4$, then $G_1 \in \Cc(\eg)$.
  If $\eeta(G_1, G_2) = 1$, then $\M(G_1) = \dc1 \egp \cup \dc2 \eg$. By Lemma~\ref{lm-weak-hoppers}, $G_1 \in \Cc(\egp) \cup \Hw(\eg)$.
  If $\eeta(G_1, G_2) = 3$, then $\M(G_1) = \dc1 \eg \cup \dc2 \egp$. By Lemma~\ref{lm-weak-hoppers}, $G_1 \in \Cc(\eg) \cup \Hw(\egp)$.
  Finally, if $\eeta(G_1, G_2) = 2$, then $\M(G_1) = \dc1 \eg \cup \dc1 \egp$. By Lemma~\ref{lm-cascades}, $G_1 \in \Cc(\eg) \cup \Cc(\egp) \cup \S$.

  Assume now that (i)--(v) hold.
  Since $\M(G) = \dc1 \eg \cup \dc1 \egp$ for $G \in \Cc(\eg) \cup \Cc(\egp) \cup \S \cup \Hw(\eg) \cup \Hw(\egp)$,
  Lemma~\ref{lm-no-cutedges} asserts that $\mu G_1$ is connected for each $\mu \in \M(G_1)$.
  Suppose first that $G_1 \in \Cc(\eg)$.
  Since $\M(G) = \dc1 \eg$, we obtain for each $\eeta(G_1, G_2) \in \{2,3,4\}$ that $G_1$ is $\eg$-tight by Lemma~\ref{lm-part-eg}.
  A similar argument works if $G_1 \in \Cc(\egp)$ and $\eeta(G_1, G_2) \in \{0, 1, 2\}$.
  If $\eeta(G_1, G_2) = 1$ and $G_1 \in \Hw(\eg)$, then  $\M(G_1) = \dc1 \egp \cup \dc2 \eg$ and $G_1$ is $\eg$-tight by Lemma~\ref{lm-part-eg}.
  If $\eeta(G_1, G_2) = 3$ and $G_1 \in \Hw(\egp)$, then $\M(G_1) = \dc2 \egp \cup \dc1 \eg$ and $G_1$ is $\eg$-tight by Lemma~\ref{lm-part-eg}.
  If $\eeta(G_1, G_2) = 2$ and $G_1 \in \S$, then $\M(G_1) = \dc1 \eg \cup \dc1 \egp$ by (C1) and $G_1$ is $\eg$-tight by Lemma~\ref{lm-part-eg}.
  This completes the proof since $\eeta(G_1, G_2) \in \{0,\ldots,4\}$ and we have proven that $G_1$ is $\eg$-tight in each case given by (i)--(v).
\end{proof}

The outcome of Theorem~\ref{th-general-eg} is summarized in Table~\ref{tb-general-eg}.
There is an analogous theorem for $\egp$-tight parts of 2-sums.

\begin{theorem}
  \label{th-general-egp}
  Let $G$ be the $xy$-sum of connected graphs $G_1, G_2 \in \Gcxy$.
  The subgraph $G_1$ is $\egp$-tight in $G$ if and only if $G_1 \in \Cc(\egp)$.
\end{theorem}

\begin{proof}
  By Lemmas~\ref{lm-part-egp} and~\ref{lm-no-cutedges}, $G_1$ is $\egp$-tight if and only if $\M(G_1) = \dc1 \egp$.
  By definition, $G_1 \in \Cc(\egp)$ if and only if $\M(G_1) = \dc1 \egp$.
\end{proof}

%------------------------------------------ Euler genus 2 ----------------------------------------------------------
\section{Excluded minors for Euler genus 2}
\label{sc-klein-euler}

\begin{figure}
  \centering
  \includegraphics{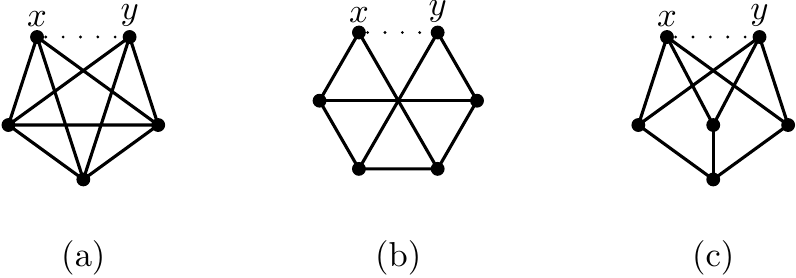}
  \caption{The class $\Cc_0(\egp)$, the third graph is the sole member of the class $\Cc_0(\eg)$.}
  \label{fg-cc-0-egp}
\end{figure}

In this section, we determine the classes $\Cc_2(\eg)$, $\Cc_2(\egp)$, and $\E_2$.
We begin by showing that the classes $\Cc_0(\eg)$ and $\Cc_0(\egp)$ are related to Kuratowski graphs $K_5$ and $K_{3,3}$.

\begin{lemma}
\label{lm-cc-0}
  The class $\Cc_0(\eg)$ consists of a single graph that is isomorphic to $K_{3,3}$ with non-adjacent terminals (Fig.~\ref{fg-cc-0-egp}c).
  The class $\Cc_0(\egp)$ consists of the three graphs shown in Fig.~\ref{fg-cc-0-egp}.
\end{lemma}

\begin{proof}
  A graph has Euler genus greater than 0 if and only if it is non-planar.
  Since both $K_5$ and $K_{3,3}$ embed into projective plane, $\E_0 = \Forb(\SS_0) = \{K_5, K_{3,3}\}$.
  By Lemma~\ref{lm-cc-eg},  a graph $G$ belongs to $\Cc_0(\eg)$ if only if $\hat{G} \in \E$.
  Since $xy \not\in E(G)$, $\hat{G}$ is not isomorphic to $K_5$ and
  thus $\Cc_0(\eg)$ consists of the unique graph isomorphic to $K_{3,3}$ with two non-adjacent terminals.

  Let us show first that each graph in Fig.~\ref{fg-cc-0-egp} belongs to $\Cc_0(\egp)$.
  If $\hat{G\+}$ is isomorphic to a Kuratowski graph, then $G \in \Cc_0(\egp)$ by Lemma~\ref{lm-cc-egp-iff}.
  Otherwise $\hat{G}$ is isomorphic to $K_{3,3}$ with $x$ and $y$ non-adjacent.
  It suffices to show that $\mu G\+$ is planar for each minor-operation $\mu \in \M(G)$ as $G\+$ clearly embeds into the projective plane.
  Pick an arbitrary edge $e \in E(G)$.
  The graph $G\+ - e$ has 9 edges and is not isomorphic to $K_{3,3}$ as it contains a triangle.
  The graph $G\+ / e$ has only 5 vertices and (at most) 9 edges.
  Since $e$ was arbitrary, it follows that $\mu G\+$ is planar for every $\mu \in \M(G)$.
  We conclude that $G \in \Cc_0(\egp)$.

  We shall show now that there are no other graphs in $\Cc_0(\egp)$.
  Let $G \in \Cc_0(\egp)$. By Lemma~\ref{lm-cc-egp}, there is a graph $H \in \Forb^*(\SS_0)$ such that either $\hat{G}$ is isomorphic to $H$
  or $G$ is isomorphic to the graph obtained from $H$ by deleting an edge and making the ends terminals.
  It is not hard to see that this yields precisely the graphs in Fig.~\ref{fg-cc-0-egp}.
\end{proof}

Note that the first two graphs in Fig.~\ref{fg-cc-0-egp} have $\et$ equal to 1 and the last one has $\et$ equal to~0.
We summarize the properties of graphs in $\Cc_0(\egp)$ in the following lemma.

\begin{lemma}
  \label{lm-cc-0-prop}
  For every $G \in \Cc_0(\egp)$, $G /xy$ is planar, $\et(G) \le 1$, and
  $\et(G) = 1$ if and only if $G \not\in \Cc_0(\eg)$.
\end{lemma}

%% \begin{proof}
%%   By Lemma~\ref{lm-cc-0}, $\hat{G}$ or $\hat{G\+}$ is isomorphic to a Kuratowski graph.
%%   For each Kuratowski graph $G$ and two vertices $x$ and $y$, the graph $G /xy$ has at most 5 vertices and at most 9 edges.
%%   Thus $G/xy$ contains no Kuratowski graph as a minor and is therefore planar.
%% \end{proof}

Let us now consider the classes $\Cc_1(\eg)$ and $\Cc_1(\egp)$.
Since a graph embeds into the projective plane if and only if it has Euler genus at most 1,
we have that $\E_1 = \Forb(\NN_1)$.
Lemma~\ref{lm-cc-eg} says that $\Cc_1(\eg)$ can be constructed from the graphs $G$ in $\E_1$ with $\eg(G) = 2$ by choosing two nonadjacent vertices as terminals.
Actually, each graph $G \in \E_1$ has $\eg(G) = 2$.
This construction yields 195 (labeled) graphs in $\Cc_1(\eg)$ and confirms that the list is complete.
Note that while there are 35 graphs in $\E_1$, the class $\Cc_1(\eg)$ is larger because graphs in $\Gxy$
have two labeled terminals.

Lemma~\ref{lm-cc-egp} provides a mean for constructing the class $\Cc_1(\egp)$.
We construct a slightly larger class and then test which of the obtained graphs are in $\Cc_1(\egp)$.
Let $G \in \Cc_1(\egp)$. If $\et(G) = 0$, then $G \in \Cc_1(\eg)$ and thus $\hat{G} \in \E_1 \ss \E^*_1$.
If $\et(G) > 0$, then $\hat{G\+} \in \E^*_1$. The class $\E^*_1$ contains 103 graphs (see~\cite{archdeacon-1981}).
Let $\A$ be the class of graphs with terminals obtained from $\E^*_1$ by either making two nonadjacent
vertices terminals or deleting an edge $e$ and making the ends of $e$ terminals.
By Lemma \ref{lm-cc-egp}, we have that $\Cc_1(\egp)\subseteq \A$.
In order to construct $\Cc_1(\egp)$, it is sufficient to check which graphs $G$ in $\A$ are minor-minimal graphs
such that $G\+$ does not embed into the projective plane.
This construction gives 250 such graphs, out of which only 227 graphs have $G\+$ 2-connected.
The intersection $\Cc_1(\eg) \cap \Cc_1(\egp)$ contains 95 graphs, so we have 132 graphs in
$\Cc_1(\egp) \setminus \Cc_1(\eg)$.

The class $\S_1$ is determined in Sect.~\ref{sc-s1} and~\ref{sc-ext} and shown to contain $21$ graphs (and all have $G\+$ 2-connected). %\X, \X = 21

By considering all 348 graphs in $\Cc_1(\eg) \cup \Cc_1(\egp) \cup \S_1$,
we obtained the following result by using computer.

\begin{lemma}
  \label{lm-all-in-klein}
  For every $G \in \Cc_1(\eg) \cup \Cc_1(\egp) \cup \S_1$, the graph $G\+$ embeds into the Klein bottle.
\end{lemma}

To prove Lemma~\ref{lm-all-in-klein}, it is sufficient to provide an embedding of $\hat{G\+}$ in the Klein bottle for each $G \in \Cc_1(\eg) \cup \Cc_1(\egp) \cup \S_1$.
The graphs and their embeddings in the Klein bottle are available online\footnote{Embeddings of $G\+$ in the Klein bottle for every $G \in \Cc_1(\eg) \cup \Cc_1(\egp) \cup \S_1$ are listed at \url{***arxiv.com}}.
Based on this evidence, we obtain the following properties of graphs in $\Cc_1(\eg)$.

\begin{lemma}
\label{lm-cc-1-eg}
  For every $G \in \Cc_1(\eg)$, we have that $\et(G) = 0$ and $\dc2 \eg \ss \dc1 \egp$.
\end{lemma}

\begin{proof}
  By Lemma~\ref{lm-all-in-klein}, $\egp(G) = \eg(G\+) \le 2$.
  Since $\eg(G) = 2$, we have that $\et(G) = \egp(G) - \eg(G) = 0$.

  The claim that $\dc2 \eg \ss \dc1 \egp$ was checked by computer.
  It is enough to show that for each $\mu \in \M(G)$ such that $\mu G$ is planar,
  the graph $\mu G\+$ is projective planar.
\end{proof}

The class of hoppers is mysterious. Although we were not able to construct any, we believe that they appear when the genus is large. However, there are none when genus is small.

\begin{lemma}
  \label{lm-hoppers}
  The classes $\Hw_1(\eg), \Hw_1(\egp)$, $\H_1(\eg)$, and $\H_1(\egp)$ are empty.
\end{lemma}

\begin{proof}
  Let $G \in \H_1(\eg)$.
  Since $G$ is non-planar, it has a Kuratowski graph $K$ as a minor. Since $\eg(G) = 2$, $K$ is a proper minor of $G$.
  Hence there is a minor-operation $\mu \in \M(G)$ such that $\mu G$ still has $K$ as a minor.
  Thus $\eg(\mu G) \ge \eg(K) = 1$. We conclude that $\mu \not\in \dc2 \eg$, a contradiction.

  Similarly, let $G \in \H_1(\egp)$. Then $\eg(G\+)=2$, and thus there is a Kuratowski graph $K$ that is a proper minor of $G\+$.
  Thus there is a minor-operation $\mu \in \M(\hat{G\+})$ such that $\hat{\mu G\+}$ has $K$ as a minor.
  Furthermore, since $\eg(K + uv) = 1$ for all $u,v \in V(G)$ by Lemma~\ref{lm-cc-0}, we may
  pick $\mu$ that does not delete nor contract $xy$. Thus $\mu \in \M(G)$.
  We have that $\egp(\mu G) \ge \eg(K) = 1$, a contradiction.

  Let $G \in \Hw_1(\egp)$. Thus $\egp(G) = 2$ and $\et(G) = 2$. Since $\eg(G) = 0$, we have that $\dc1 \eg = \emptyset$.
  We conclude that $\M(G) = \dc2 \egp$.
  Hence $G \in \H_1(\egp)$ which was already shown to be empty.

  Let $G \in \Hw_1(\eg)$. Thus $\egp(G) = 2$, $\et(G) = 0$, and $G \in \Cc_1(\eg)$.
  By Lemma~\ref{lm-cc-1-eg}, $\M(G) = \dc1 \egp$. Thus $G \in \Cc(\egp)$, a contradiction.
\end{proof}

\begin{table}
  \centering
  \begin{tabular}{|c | c |}
    \hline
    $\eeta(G_1, G_2)$ & $G_1$ \\
    \hline
    0 & $\Cc(\egp)$ \\
    1 & $\Cc(\egp)$ \\
    2 & $\Cc(\egp) \cup \Cc(\eg) \cup \S$ \\
    \hline
  \end{tabular}
  \caption{Classification of $\eg$-tight parts of a 2-sum in $\Cc_2(\eg)$.}
  \label{tb-klein}
\end{table}

Let us now state some properties of the parts of $xy$-sums in $\Cc_2(\eg)$ and $\Cc_2(\egp)$.

\begin{lemma}
  \label{lm-bounds-eg}
  Let $G$ be the $xy$-sum of connected graphs $G_1, G_2 \in \Gcxy$ such that~$\egp(G_1) \le \egp(G_2)$.
  If $G \in \Cc_2(\eg)$, then
  \begin{enumerate}[label=\rm(\roman*)]
  \item
    $\egp(G_1) = 1$,
  \item
    $\egp(G_2) = 2$,
  \item
    $\eeta(G_1, G_2) \le 2$.
  \end{enumerate}
\end{lemma}

\begin{proof}
  If $\egp(G_2) > 2$, then since $G_2^+$ is a proper minor of $G$, there is a minor-operation $\mu \in \M(G)$ such that $\eg(\mu G) \ge \eg(G_2^+) > 2$,
  a contradiction. Thus $\egp(G_2) \le 2$.
  If $\egp(G_1) = 0$, then $\eg(G) \le \eh_1(G) = \egp(G_1) + \egp(G_2) \le 2$ by Theorem~\ref{th-richter-euler}, a contradiction.
  Hence $\egp(G_1) \ge 1$.

  By Theorem~\ref{th-richter-euler}, we have
  $$\eh_1(G) = \egp(G_1) + \egp(G_2) \ge \eg(G) = 3.$$
  This implies that $\egp(G_2) = 2$ and (ii) holds.
  We also have
  $$\eh_0(G) = \eg(G_1) + \eg(G_2) + 2 \ge \eg(G) = 3.$$
  Therefore, $\eg(G_1) + \eg(G_2) \ge 1$.

  Suppose that $\egp(G_1) = 2$. If $\eg(G_1) + \eg(G_2) \ge 2$, then by Theorem~\ref{th-richter-euler},
  $$\eg(G) = \min\{\eh_0(G), \eh_1(G)\} = 4,$$
  a contradiction with $\eg(G) = 3$.
  Hence $\eg(G_1) + \eg(G_2) = 1$.
  Since $\egp(G_1) = \egp(G_2)$, we may exchange the roles of $G_1$ and $G_2$ if necessary and thus assume that $\eg(G_1) = 0$.
  By Lemma~\ref{lm-hoppers}, $\H_1(\egp) = \emptyset$ and thus there exists a minor-operation $\mu \in \M(G_1)$ such that $\egp(\mu G_1) \ge 1$.
  Note that $\eg(\mu G_1) = 0$.
  By Theorem~\ref{th-richter-euler},
  $$\eg(\mu G) = \min\{\eh_0(\mu G),\eh_1(\mu G)\} = \min\{\eg(\mu G_1) + \eg(G_2) + 2, \egp(\mu G_1) + \egp(G_2)\} = 3,$$
  a contradiction with $G \in \Cc_2(\eg)$.
  We conclude that $\egp(G_1) = 1$ and (i) holds.
  Since $\eg(G_1) + \eg(G_2) \ge 1$ and $\egp(G_1) + \egp(G_2) = 3$, we have that $\eeta(G_1, G_2) \le 2$ and (iii) holds.
\end{proof}

\begin{lemma}
  \label{lm-bounds-egp}
  Let $G$ be the $xy$-sum of connected graphs $G_1, G_2 \in \Gcxy$ such that~$\egp(G_1) \le \egp(G_2)$.
  If $G \in \Cc_2(\egp)$, then
  \begin{enumerate}[label=\rm(\roman*)]
  \item
    $\egp(G_1) = 1$,
  \item
    $\egp(G_2) = 2$.
  \end{enumerate}
\end{lemma}

\begin{proof}
  If $\egp(G_2) > 2$, then, since $G_2$ is a proper minor of $G$, there is a minor-operation $\mu \in \M(G_1)$,such that
  $\mu G$ still has $G_2$ as a minor. Hence $\egp(\mu G) \ge \egp(G_2) > 2$. We conclude that $\egp(G) = \egp(\mu G) = 3$, a contradiction. This shows that $\egp(G_2)\le2$.

  By Theorem~\ref{th-richter-euler}, we have
  $$3 = \egp(G) = \eh_1(G) = \egp(G_1) + \egp(G_2).$$
  Since $\egp(G_2) \le 2$, we conclude that $\egp(G_1) = 1$ and $\egp(G_2) = 2$.
  Thus (i) and (ii) hold.
\end{proof}

Finally, we are ready to state a theorem which classifies the $xy$-sums in $\Cc_2(\eg)$.

\begin{theorem}
\label{th-cc-2-eg}
  Let $G$ be the $xy$-sum of connected graphs $G_1, G_2 \in \Gcxy$.
  If the following statements {\rm(i)--\rm(iv)} hold, then $G \in \Cc_2(\eg)$.
  \begin{enumerate}[label=\rm(\roman*)]
  \item
    $G_1 \in \Cc_0(\egp)$.
  \item
    $G_2 \in \Cc_1(\egp) \cup \S_1$.
  \item
    If $G_1 \in \Cc_0(\eg)$, then $G_2 \in \Cc_1(\egp)$.
  \item
    If $G_1 \not\in \Cc_0(\eg)$, then $\et(G_2) \le 1$.
  \end{enumerate}
  Conversely, every 2-connected graph $G \in \Cc_2(\eg)$ such that $\{x,y\}$ is a 2-vertex-cut can be obtained in this way.
\end{theorem}

\begin{proof}
  Suppose that statements (i)--(iv) hold. Our goal is to show that $G \in \Cc_2(\eg)$.
  By Lemma~\ref{lm-tight}, it is enough to prove that $G_1$ and $G_2$ are $\eg$-tight in $G$ and that $\eg(G) = 3$.
  If $G_1 \in \Cc_0(\eg)$, then $\et(G_1) = 0$ by Lemma~\ref{lm-cc-0-prop}. Otherwise, $\et(G_1) = 1$ and $\et(G_2) \le 1$ by (iv).
  We conclude that in both cases we have $\eeta(G_1, G_2) \le 2$.
  Theorem~\ref{th-general-eg} and (i) give that $G_1$ is $\eg$-tight in $G$.
  If $\eeta(G_1, G_2) = 2$, then $G_2$ is $\eg$-tight in $G$ by Theorem~\ref{th-general-eg} and (ii).
  Suppose now that $\eeta(G_1, G_2) \le 1$ and $G_2 \in \S_1$.
  Since $\et(G_2) = 1$ by Lemma~\ref{lm-cascades}, we have that $\et(G_1) = 0$ and hence $G_1 \in \Cc_0(\eg)$ by Lemma~\ref{lm-cc-0-prop}.
  This is a contradiction with (iii).
  Thus, we may assume that $G_2 \in \Cc_1(\egp)$. Theorem~\ref{th-general-eg} asserts that $G_2$ is $\eg$-tight in $G$.
  Since $\eeta(G_1, G_2) \le 2$, $\egp(G_1) = 1$, and $\egp(G_2) = 2$, Theorem~\ref{th-richter-euler} and~(\ref{eq-eta-egp}) give that
  $$\eg(G) = \eh_1(G) = \egp(G_1) + \egp(G_2) = 3.$$
  Therefore, $G \in \Cc_2(\eg)$.

  We shall now show the converse, that is, for $G \in \Cc_2(\eg)$ where $\{x,y\}$ is a 2-vertex-cut, we find connected graphs $G_1, G_2 \in \Gcxy$ such that
  $G$ is an $xy$-sum of $G_1$ and $G_2$ and (i)--(iv) hold.
  Let us distribute the $\{x, y\}$-bridges arbitrarily into $G_1$ and $G_2$ so that $\egp(G_1) \le \egp(G_2)$
  and $G_1, G_2$ contain at least one of the bridges.
  By Lemma~\ref{lm-bounds-eg}, we have that $\egp(G_1) = 1$, $\egp(G_2) = 2$, and $\eeta(G_1, G_2) \le 2$.
  Since $\Hw_0(\eg)$ and $\S_0$ are empty (see Lemma~\ref{lm-cascades}),
  Theorem~\ref{th-general-eg} gives that $G_1 \in \Cc_0(\eg) \cup \Cc_0(\egp) = \Cc_0(\egp)$.
  Thus (i) holds.

  Since $\eeta(G_1, G_2) \le 2$, $G_2 \in \Cc(\eg) \cup \Cc(\egp) \cup \S \cup \Hw(\eg)$ by Theorem~\ref{th-general-eg}.
  By Lemma~\ref{lm-hoppers}, $\Hw_1(\eg)$ is empty.
  Since $\egp(G_2) = 2$, we have that $G_2 \not\in \Cc_0(\eg) \cup \Cc_0(\egp)$.
  We conclude that $G_2 \in \Cc_1(\eg) \cup \Cc_1(\egp) \cup \S_1$.
  Assume for a contradiction that $G_2 \in \Cc_1(\eg) \sm \Cc_1(\egp)$.
  Thus there exists a minor-operation $\mu \in \M(G_2)$ such that $\mu \not\in \dc2 \eg \cup \dc1 \egp$ since $\Hw_1(\eg)$ is empty. %by Lemma~\ref{lm-eg-sm-egp}.
  By~(\ref{eq-eg-1}), $\eh_1(\mu G) = \eh_1(G)$. Since $G_2$ is $\eg$-tight in $G$, Theorem~\ref{th-richter-euler} gives:
  $$3 = \eg(G) > \eg(\mu G) = \eh_0(\mu G) = \eg(G_1) + \eg(\mu G_2) + 2 \ge 3.$$
  This contradicts our assumption that $G_2 \not\in \Cc_1(\egp) \cup \S_1$.
  We conclude that (ii) holds.

  Suppose that $G_1 \in \Cc_0(\eg)$ and $G_2 \in \S_1$.
  Since $\et(G_1) = 0$ and $\et(G_2) = 1$ by Lemmas~\ref{lm-cascades} and~\ref{lm-cc-0-prop}, we have that $\eeta(G_1, G_2) = 1$.
  This contradicts Theorem~\ref{th-general-eg}.
  Thus (iii) holds.

  In order to show (iv), suppose that $G_1 \not\in \Cc_0(\eg)$ and $\et(G_2) = 2$.
  Then $\et(G_1) = 1$ by Lemma~\ref{lm-cc-0-prop} and thus $\eeta(G_1, G_2) = 3$.
  This contradicts Lemma~\ref{lm-bounds-eg}(iii).
  We conclude that (iv) holds.
\end{proof}

We also have a corresponding theorem that classifies the $xy$-sums in $\Cc_2(\egp)$.

\begin{theorem}
\label{th-cc-2-egp}
  Let $G$ be the $xy$-sum of connected graphs $G_1, G_2 \in \Gcxy$.
  If the following statements (i) and (ii) hold, then $G \in \Cc_2(\egp)$.
  \begin{enumerate}[label=\rm(\roman*)]
  \item
    $G_1 \in \Cc_0(\egp)$.
  \item
    $G_2 \in \Cc_1(\egp)$.
  \end{enumerate}
  Conversely, every 2-connected graph $G \in \Cc_2(\egp)$ such that $\{x,y\}$ is a 2-vertex-cut can be obtained this way.
\end{theorem}

\begin{proof}
  Suppose that (i) and (ii) hold.
  By Theorem~\ref{th-general-egp}, $G_1$ and $G_2$ are $\egp$-tight in $G$.
  Thus $G \in \Cc(\egp)$ by Lemma~\ref{lm-tight}. By Theorem~\ref{th-richter-euler},
  $$\egp(G) = \eh_1(G) = \egp(G_1) + \egp(G_2) = 3.$$
  Therefore, $G \in \Cc_2(\egp)$.

  For the converse, let $G_1$ and $G_2$ be collections of $\{x, y\}$-bridges in $G$ such that $G$ is the $xy$-sum of $G_1$ and $G_2$, $\egp(G_1) \le \egp(G_2)$, and $G_1, G_2$ contain at least one of the bridges.
  We shall show that (i) and (ii) hold.
  By Theorem~\ref{th-general-egp}, $G_1, G_2 \in \Cc(\egp)$.
  By Lemma~\ref{lm-bounds-egp}, $\egp(G_1) = 1$ and $\egp(G_2) = 2$.
  We conclude that $G_1 \in \Cc_0(\egp)$ and $G_2 \in \Cc_1(\egp)$ and thus (i) and (ii) hold.
\end{proof}

The following lemma gives necessary and sufficient conditions for the edge $xy$ to be $\eg$-tight
in a graph with a 2-vertex-cut $\{x, y\}$ and the edge $xy$.

\begin{lemma}
\label{lm-edge-xy-eg}
  Let $G$ be an $xy$-sum of connected graphs $G_1, G_2 \in \Gcxy$ and let $H = \hat{G\+}$.
  Then the subgraph of $H$ consisting of the edge $xy$ is $\eg$-tight in $H$
  if and only if $\eeta(G_1, G_2) > 2$ and either $\eg(G_1 /xy) < \egp(G_1)$ or $\eg(G_2 / xy) < \egp(G_2)$,
\end{lemma}

\begin{proof}
  Since $\eg(H) = \egp(G) = \eg(G) + \et(G)$ and $\eg(H - xy) = \eg(G)$, we have that
  $\eg(H - xy) < \eg(H)$ if and only if $\et(G) > 0$.
  Theorem~\ref{th-richter-euler} gives that $\et(G) > 0$ if and only if $\eeta(G_1, G_2) > 2$.
  Thus we may assume below that $\eeta(G_1, G_2) > 2$.

  By Theorem~\ref{th-stahl-euler}, $\eg(H /xy) = \eg(G_1/xy) + \eg(G_2/xy)$.
  Since $\eg(G_1/xy) \le \egp(G_1)$ and $\eg(G_2 /xy) \le \egp(G_2)$, we have that
  $\eg(H/xy) < \eg(H)$ if and only if either $\eg(G_1/xy) < \egp(G_1)$ or $\eg(G_2/xy) < \egp(G_2)$.
\end{proof}

We conclude this section by characterizing the graphs of connectivity 2 in $\E_2$.

\begin{theorem}
  \label{th-e-2}
  Let $G$ be an $xy$-sum of connected graphs $G_1, G_2 \in \Gcxy$ such that the following holds:
  \begin{enumerate}[label=\rm(\roman*)]
  \item
    $G_1 \in \Cc_0(\egp)$.
  \item
    $G_2 \in \Cc_1(\egp) \cup \S_1$.
  \item
    If $G_1 \in \Cc_0(\eg)$, then $G_2 \in \Cc_1(\egp)$.
  \end{enumerate}
  If $\eeta(G_1, G_2) \le 2$, then $\hat{G} \in \E_2$.
  If $\eeta(G_1, G_2) > 2$, then $\hat{G\+} \in \E_2$.
  Furthermore, each graph in $\E_2$ of connectivity 2 is constructed this way.
\end{theorem}

\begin{proof}
  Assume first that $\eeta(G_1, G_2) \le 2$.
  By Lemma~\ref{lm-cc-eg}, it is sufficient to show that $G_1$ and $G_2$ satisfy the conditions (i)--(iv) of Theorem~\ref{th-cc-2-eg}.
  The conditions (i)--(iii) of Theorem~\ref{th-cc-2-eg} are the same as the assumptions of this theorem.
  If $G_1 \not\in \Cc_1(\eg)$, then $\et(G_1) = 1$ by Lemma~\ref{lm-cc-0-prop}.
  Since $\eeta(G_1, G_2) \le 2$, we have that $\et(G_2) \le 1$ and (iv) holds.
  By Theorem~\ref{th-cc-2-eg}, $G \in \Cc_2(\eg)$. By Lemma~\ref{lm-cc-eg}, $\hat{G} \in \E_2$.

  Assume now that $\eeta(G_1, G_2) > 2$.
  Since, for each graph $G \in \Cc_0(\egp) \cup \S_1$, $\et(G) \le 1$, by Lemmas~\ref{lm-cascades} and~\ref{lm-cc-0-prop},
  we conclude that $\eeta(G_1, G_2) = 3$, $\et(G_1) = 1$, $\et(G_2) = 2$, $G_1 \not\in \Cc_0(\eg)$, and $G_2 \in \Cc_1(\egp)$.
  By Theorem~\ref{th-cc-2-egp}, $G \in \Cc_2(\egp)$.
  Note that this implies that $\hat{G}$ is $\eg$-tight in $\hat{G\+}$.
  Since $\egp(G_1/xy) < \egp(G_1)$ (Lemma~\ref{lm-cc-0-prop}), we obtain that $xy$ is $\eg$-tight in $\hat{G\+}$ by Lemma~\ref{lm-edge-xy-eg}.
  Since $\eg(\hat{G\+}) = \egp(G) = 3$, $\hat{G\+} \in \E_2$ by Lemma~\ref{lm-tight}.

  Let us now prove that each $H \in \E_2$ of connectivity 2 is constructed this way.
  Pick an arbitrary 2-vertex-cut $\{x, y\}$ of $H$.
  Suppose first that $xy \in E(H)$.
  Consider $G = H-xy$ as a graph in $\Gcxy$.
  Since $\M(G) \ss \M(H)$, we have that $\M(G) = \dc1 \egp$ and $G \in \Cc(\egp)$.
  Suppose that $\egp(G) > 3$. Let $G_1$ and $G_2$ be parts of $G$ such that $\egp(G_1) \le \egp(G_2)$.
  If $\egp(G_2) > 2$, then for any minor-operation $\mu \in \M(G_1)$, the graph $\mu G$ has $G_2^+$ as a minor.
  Hence $\eg(\mu H) \ge \egp(G_2) > 2$, a contradiction. Therefore, $\egp(G_2) \le 2$.
  By Theorem~\ref{th-richter-euler}, $\egp(G_1) = \egp(G_2) = 2$.
  Let $\mu \in \M(G_1)$. By Theorem~\ref{th-richter-euler},
  $2 \ge \eg(\mu H) = \egp(\mu G) = \egp(\mu G_1) + \egp(G_2)$.
  Hence $\egp(\mu G_1) = 0$. We conclude that $\M(G_1) = \dc2 \egp$ and $G_1 \in \H_1(\egp)$.
  By Lemma~\ref{lm-hoppers}, $\H_1(\egp)$ is empty, a contradiction.

  So we may assume that $\egp(G) = 3$ and thus $G \in \Cc_2(\egp)$.
  By Theorem~\ref{th-cc-2-egp}, $G$ is an $xy$-sum of graphs $G_1 \in \Cc_0(\egp)$ and $G_2 \in \Cc_1(\egp)$.
  By Lemma~\ref{lm-edge-xy-eg}, $\eeta(G_1, G_2) > 2$.
  Thus $G$ satisfies the conditions (i)--(iii) of the theorem.

  Suppose now that $xy \not\in E(H)$. Consider $G = H$ as a graph in $\Gcxy$.
  By Lemma~\ref{lm-cc-eg}, $G \in \Cc(\eg)$.
  Suppose that $\eg(G) > 3$. Let $G_1$ and $G_2$ be parts of $G$ such that $\egp(G_1) \le \egp(G_2)$.
  If $\egp(G_2) > 2$, then for any minor-operation $\mu \in \M(G_1)$ so that $\mu G_1$ is connected, the graph $\mu G$ has $G_2^+$ as a minor.
  Hence $\eg(\mu H) \ge \egp(G_2) > 2$, a contradiction. Therefore, $\egp(G_2) \le 2$.
  By Theorem~\ref{th-richter-euler}, $\egp(G_1) = \egp(G_2) = 2$ and $3 < \eg(G) \le \eh_0(G) = \eg(G_1) + \eg(G_2) + 2$.
  We may assume that $\eg(G_1) \le \eg(G_2)$ and so $\eg(G_2) \ge 1$.
  Let $\mu \in \M(G_1)$. By Theorem~\ref{th-richter-euler},
  $$2 \ge \eg(\mu H) = \eg(\mu G) = \min\{\eh_0(\mu G), \eh_1(\mu G)\}.$$
  Since $\eh_0(\mu G) = \eg(\mu G_1) + \eg(G_2) + 2 \ge 3$, we have that
  $2 \ge \eh_1(\mu G) = \egp(\mu G_1) + \egp(G_2)$.
  We conclude that $\egp(\mu G_1) = 0$ and $\mu \in \dc2 \egp$. Since $\mu$ was arbitrary, $G_1 \in \H_1(\egp)$.
  This contradicts Lemma~\ref{lm-hoppers} which asserts that $\H_1(\egp)$ is empty.

  Thus we may assume that $\eg(G) = 3$ and thus $G \in \Cc_2(\eg)$.
  By Theorem~\ref{th-cc-2-eg}, $G$ is an $xy$-sum of graphs $G_1 \in \Cc_0(\egp)$ and $G_2 \in \Cc_1(\egp) \cup \S_1$
  and either $G_1 \not\in \Cc_0(\eg)$ or $G_2 \not\in \S_1$.
  If $G_1 \in \Cc_0(\eg)$, then $\et(G_1) = 0$ by Lemma~\ref{lm-cc-0-prop} and thus $\eeta(G_1, G_2) \le 2$.
  Otherwise, $\et(G_1) \le 1$ and $\et(G_2) \le 1$ by Theorem~\ref{th-cc-2-eg}(iv) and we obtain that $\eeta(G_1, G_2) \le 2$.
  Thus $G$ satisfies the conditions (i)--(iii) of the theorem.
\end{proof}

As a corollary we can construct the complete list of graphs in $\E_2$ of connectivity~2.

\begin{corollary}
  \label{cr-e-2}
  There are precisely $668$ graphs of connectivity 2 that are critical for Euler genus 2. %3*(227 - 39) + 2*(\X - 4)  + 70, \X = 21
\end{corollary}

\begin{proof}
  Let us begin by counting the number of pairs $G_1, G_2$ that satisfy the conditions (i)--(iii) of Theorem~\ref{th-e-2}.
  There are 3 graphs in $\Cc_0(\egp)$, there are 227 graphs $G_2$ in $\Cc_1(\egp)$ such that $G_2^+$ is 2-connected, and there are $21$ graphs in $\S_1$ (for each $G \in \S_1$, the graph $G^+$ is 2-connected). %\X = 21
  That gives $744$ pairs since $|\Cc_0(\eg)|=1$ and $|\S_1|=21$. %3(227 + \X)
  There are only $744-21=723$ pairs that satisfy the condition (iii) of Theorem~\ref{th-e-2} that either $G_1 \not\in \Cc_0(\eg)$ or $G_2 \not\in \S_1$.  %3*227 + 2\X

  Let $G_1, G_2 \in \Gcxy$.
  There are two $xy$-sums that have parts isomorphic to $G_1$ and $G_2$ as there are two ways how to identify two pairs of vertices.
  If $G_1 \in \Cc_0(\egp)$, then there is an automorphism of $G_1$ exchanging the terminals. Hence there is only a single non-isomorphic $xy$-sum $G$
  that has parts $G_1 \in \Cc_0(\egp)$ and $G_2 \in \Cc_1(\egp) \cup \S_1$.
  Since $\eeta(G_1, G_2)$ depends only on $G_1$ and $G_2$, precisely one of $\hat{G}, \hat{G\+}$ belongs to $\E_2$.
  There may be more pairs $G_1, G_2$ giving the same graph $H \in \E_2$ though.

  Let $H \in \E_2$ have connectivity~2. By Theorem~\ref{th-e-2}, there exists an $xy$-sum $G$ of connected graphs $G_1$ and $G_2$
  such that either $\hat{G} \iso H$ or $\hat{G\+} \iso H$.
  Note that $G_1^+$ and $G_2^+$ are 2-connected.
  Suppose that $H$ admits a nontrivial automorphism $\psi$ such that $\psi(V(G_1)) \not= V(G_1)$ (otherwise, it is just a combination of two automorphisms of $G_1$ and $G_2$).
  It is not hard to see that if $G_2^+$ is 3-connected, each automorphism of $H$ is trivial.
  Therefore, we need to study graphs $G_2 \in \Cc_1(\egp) \cup \S_1$ such that $G_2^+$ has connectivity~2.

  There are 39 graphs $G_2$ in $\Cc_1(\egp)$ such that $G_2^+$ has connectivity 2 and there are 4 graphs $G_2$ in $\S_1$ such that $G_2^+$ has connectivity 2 (see Fig.~\ref{fg-cascades-1}).
  It is not hard to check that the $125$ pairs %39*3 + 4*2
  with $G_1 \in \Cc_0(\egp)$ make only 70 non-isomorphic graphs in $\E_2$. %70?
  We conclude that there are $668$ graphs of connectivity 2 in $\E_2$. %3*(227 - 39) + 2*(\X - 4)  + 70
\end{proof}

%%---------------------------------- Klein bottle -------------------------------------------------------
\section{The Klein bottle}
\label{sc-klein-klein}

In this section, we characterize the obstructions of connectivity 2 for embedding graphs into the Klein bottle.
Let us introduce graph parameters $\os$ and $\osp$ that capture the property of being orientably simple.
Let $\os = \ng - \eg$ and let $\osp = \ngp - \egp$.
Note that $\os(G) = 1$ if $G$ is orientably simple and $\os(G) = 0$ otherwise.
%Also note that $\eg$ and $\ng$ are $1$-separated by $\os$ and $\egp$ and $\ngp$ are $1$-separated by $\osp$.

The following lemma is an easy consequence of Lemma~\ref{lm-ng-upper}.

\begin{lemma}
\label{lm-non-simple}
If\/ $\egp(G)$ is odd, then $\osp(G) = 0$.
\end{lemma}
Let us state the following theorem of Stahl and Beineke using our formalism.

\begin{theorem}[Stahl and Beineke~\cite{stahl-1977}]
\label{th-stahl-nonori}
  Let $G = G_1 \cup G_2$ be a $1$-sum of $G_1$ and $G_2$. Then
  $$\ng(G) = \eg(G_1) + \eg(G_2) + \os(G_1)\os(G_2).$$
  Moreover, $\os(G) = \os(G_1)\os(G_2)$.
\end{theorem}

In order to describe how the nonorientable genus of a 2-sum of graphs can be computed from
the genera of its parts, let us introduce parameters $\nh_0$ and $\nh_1$ similar to $\eh_0$
and $\eh_1$.
Let $G$ be an $xy$-sum of connected graphs $G_1, G_2 \in \Gxy$.
Define
\begin{equation}
  \nh_0(G) = \eh_0(G) = \eg(G_1) + \eg(G_2) + 2\label{eq-ng-0}
\end{equation}
and
\begin{equation}
  \nh_1(G) = \egp(G_1) + \egp(G_2) + \osp(G_1)\osp(G_2).\label{eq-ng-1}
\end{equation}

Let $\nt = \ngp - \ng$.
We shall use the following theorem of Richter.

\begin{theorem}[Richter~\cite{richter-1987-nonori}]
  \label{th-richter-nonori}
  Let $G$ be an $xy$-sum of connected graphs $G_1, G_2 \in \Gxy$.
  Then
  \begin{enumerate}[label=\rm(\roman*)]
  \item
    $\ng(G) = \min\{\nh_0(G), \nh_1(G)\}$,
  \item
    $\ngp(G) = \nh_1(G)$,
  \item
    $\nt(G) = \max\{\nh_1(G) - \nh_0(G), 0\}$,
  \item
    $\osp(G) = \osp(G_1)\osp(G_2)$, and
  \item
    if $\osp(G) = 0$ or $\eeta(G_1, G_2) \ge 2$, then $\os(G) = 0$, else $\os(G) = 1$.
  \end{enumerate}
\end{theorem}

The next lemma shows that the $xy$-sums of graphs with parts that are not orientably simple are
critical graphs for Euler genus if and only if they are obstructions for the corresponding nonorientable surface.

\begin{lemma}
  \label{lm-osp-0}
  Let $G$ be an $xy$-sum of connected graphs $G_1, G_2 \in \Gcxy$, $H \in \{\hat{G}, \hat{G\+}\}$, and $k \ge 0$.
  If $\osp(G_1) = \osp(G_2) = 0$, then $H \in \E_k$ if and only if $H \in Forb(\NN_k)$.
\end{lemma}

\begin{proof}
%  Let $H \in \{\hat{G}, \hat{G\+}\}$.
  By Theorem~\ref{th-richter-nonori}(iv) and (v), $\os(G) = \osp(G) = \osp(G_1)\osp(G_2) = 0$.
  Therefore, $\os(H) = 0$.

  Assume first that $H \in \Forb(\NN_k)$.
  We have that $\eg(H) = \ng(H) > k$.
  Let $\mu \in \M(H)$. Since $\eg(\mu H) \le \ng(\mu H) \le k$ and $\mu$ is arbitrary, we have that $H \in \E_k$.

  Assume now that $H \in \E_k$.
  By Lemmas~\ref{lm-cc-eg} and~\ref{lm-cc-egp-iff}, $G \in \Cc(\eg) \cup \Cc(\egp)$.
  We have that $\ng(H) = \eg(H) > k$.
  Let $\mu \in \M(G_1)$. By Lemma~\ref{lm-no-cutedges}, $\mu G_1$ is connected.
  By Theorem~\ref{th-richter-nonori}(iv) and (v), $\os(\mu G) = \osp(\mu G) = \osp(\mu G_1)\osp(G_2) = 0$.
  Therefore, $\os(\mu H) = 0$.
  Hence $\ng(\mu H) = \eg(\mu H) \le k$.
  Similarly $\ng(\mu H) \le k$ for $\mu \in \M(G_2)$.
  This shows that $H \in \Forb(\NN_k)$ if $H = \hat{G}$.
  Assume then that $H = \hat{G\+}$. It remains to see that after deleting or contracting
  the edge $xy$, the graph can be embedded in $\NN_k$.
  Since $\os(H - xy) = \os(G) = 0$, we have that $\ng(H - xy) = \eg(H - xy) \le k$.
  By Theorem~\ref{th-stahl-nonori},
  $\os(H / xy) = \os(G_1/xy)\os(G_2/xy)$.
  If $\os(H /xy) = 0$, then
  $\ng(H /xy) = \eg(H /xy) \le k$.
  So, we are done unless $\os(G_1/xy) = \os(G_2/xy) = 1$, which we assume henceforth.
  Since $G_1/xy$ is a minor of $\hat{G_1^+}$, we have that
  $$\eg(G_1/xy) = \ng(G_1/xy) - \os(G_1/xy) < \ng(G_1/xy) \le \ng(G_1^+) = \ngp(G_1) = \egp(G_1).$$
  Similarly, we derive that $\eg(G_2/xy) < \egp(G_2)$.
  By Theorems~\ref{th-richter-euler} and~\ref{th-stahl-nonori},
  \begin{align*}
    \ng(H /xy) &\le \eg(G_1/xy) + \eg(G_2/xy) + 1 < \egp(G_1) + \egp(G_2) \\
    &= \egp(G) = \eg(H) = \ng(H) \le k.
  \end{align*}
  We conclude that $H \in \Forb(\NN_k)$.
\end{proof}

A corollary of Theorem~\ref{th-e-2} and Lemma~\ref{lm-osp-0} asserts that the class of obstructions for the Klein bottle having connectivity 2
and the class of critical graphs for Euler genus 2 of connectivity 2 are the same.
We can say even more:

\begin{corollary}
  \label{cr-klein-bottle}
  Let $H$ be a graph of connectivity 2. Then $H \in \E_2$ if and only if $H \in \Forb(\NN_2)$.
\end{corollary}

\begin{proof}
  Assume first that $H \in \E_2$.
  By Theorem~\ref{th-e-2}, there is an $xy$-sum $G$ of graphs $G_1 \in \Cc_0(\egp)$ and $G_2 \in \Cc_1(\egp) \cup \S_1$
  such that $H \in \{\hat{G}, \hat{G\+}\}$.
  By Lemma~\ref{lm-non-simple}, $\osp(G_1) = 0$.
  By Lemma~\ref{lm-all-in-klein}, $\osp(G_2) = 0$.
  By Lemma~\ref{lm-osp-0}, $H \in \Forb(\NN_2)$.

  Assume now that $H \in \Forb(\NN_2)$.
  Let $G$ be an $xy$-sum of connected graphs $G_1, G_2 \in \Gcxy$ such that $H \in \{\hat{G}, \hat{G\+}\}$.
  Suppose that $\osp(G_1) = 1$.
  If $\egp(G_1) \ge 2$, then $\ngp(G_1) \ge 3$ and thus $G_1^+$ does not embed into $\NN_2$.
  This yields a contradiction as $H$ has $G_1^+$ as a proper minor.
  Since $\egp(G_1) > 0$, we conclude that $\egp(G_1) = 1$.
  By Lemma~\ref{lm-non-simple}, $\osp(G_1) = 0$, a contradiction.
  Therefore by symmetry, $\osp(G_1) = \osp(G_2) = 0$.
  By Lemma~\ref{lm-osp-0}, $H \in \E_2$.
\end{proof}

%----------------------------------- Cascades ------------------------------------------------------------

%----------------------------- Bridges and cycles ---------------------------------------
\section{Bridges and cycles}
\label{sc-bridges-and-cycles}

In the rest of the paper, we develop framework which we use to determine the class $\S_1$ of cascades of genus 1.

Let $H$ be a subgraph of $G$. An \df{$H$-bridge} $B$ is either an edge in $E(G) \sm E(H)$ with both ends in $H$ or
a connected component $C$ of $G - V(H)$ together with all edges with at least one end in $C$.
In the former case we say that the bridge $B$ is \df{trivial}. The vertices in $V(B) \cap V(H)$ are the \df{attachments} of $B$.
We also say that $B$ \df{attaches} at $v$, for $v \in V(B) \cap V(H)$.
The graph $B^\circ = B - V(H)$ is the \df{interior} of $B$.
We will use the  following lemma (see~\cite[Prop.~6.1.2.]{mohar-book}).

\begin{lemma}
  \label{lm-planar-patch}
  Let $G_1 \in \Gcxy$ be a nontrivial $\{x, y\}$-bridge of a graph $G$.
  If $G_1^+$ is planar, then every embedding of $(G-G_1^\circ)\+$ into a surface can be extended to an embedding of $G$ into the same surface.
\end{lemma}

A \df{branch vertex} in $H$ is a vertex of degree different from 2. A \df{branch} in $H$ is a path $P$ connecting two branch vertices $v_1, v_2$ such that
all vertices in $V(P) \sm \{v_1, v_2\}$ have degree 2 in $H$. An \df{open branch} is obtained from a branch by removing its endvertices.

A \df{subdivision} of $G$ is a graph obtained from $G$ by replacing each edge of $G$ by a path of length at least 1.
A graph $H$ is \df{homeomorphic} to $G$, $H\cong G$, if there is a graph $K$ such that both $G$ and $H$ are isomorphic to subdivisions of $K$.
A \df{Kuratowski subgraph} in $G$ is a subgraph of $G$ homeomorphic to a \df{Kuratowski graph}, $K_5$ or $K_{3,3}$.
A \df{K-graph} in $G$ is a subgraph $L$ of $G$ which is homeomorphic to either $K_4$ or $K_{2,3}$
such that there is an $L$-bridge in $G$ that attaches to all four branch vertices of $L$ when $L\cong K_4$
or attaches to all three open branches of $L$ when $L\cong K_{2,3}$. Such an $L$-bridge is a \df{principal} $L$-bridge.

Let $C$ be a cycle in a graph $G$.
Two $C$-bridges $B_1$ and $B_2$ \df{overlap} if at least one of the following conditions hold:
\begin{enumerate}[label=\rm(\roman*)]
\item
  $B_1$ and $B_2$ have three attachments in common;
\item
  $C$ contains distinct vertices $v_1, v_2, v_3,  v_4$ that appear in this order on $C$ such that
  $v_1$ and $v_3$ are attachments of $B_1$ and $v_2$ and $v_4$ are attachments of $B_2$.
\end{enumerate}
In the case (ii), we say that $B_1$ and $B_2$ \df{skew-overlap}.
The \df{overlap graph} $O(G,C)$ of $G$ with respect to $C$ is the graph whose vertex-set consists of the $C$-bridges in $G$,
and two $C$-bridges are adjacent in $O(G,C)$ if they overlap.

Let $C$ be a cycle in a graph $G$.
For a $C$-bridge $B$ in $G$, the \df{$B$-side} of $C$ is the union of all $C$-bridges at even distance from $B$ in the overlap graph $O(G, C)$.
For a vertex $v \in V(G) \sm V(C)$, the \df{$v$-side} of $C$ is the $B$-side of the $C$-bridge $B$ containing $v$.
Two vertices $u,v \in V(G) \sm V(C)$ are \df{separated} by $C$ if the $C$-bridges containing $u$ and $v$ have odd distance in $O(G,C)$.
We also say that $C$ is \df{$(u,v)$-separating}.

Let $G$ be a $\Pi$-embedded graph with the set $F(\Pi)$ of $\Pi$-faces.
%% The \df{$\Pi$-face-incidence graph} $\FI_\Pi$ is the bipartite graph with parts $F(\Pi)$ and $V(G)$
%% where a face $f \in F(\Pi)$ is adjacent to a vertex $v \in V(G)$ if $f$ and $v$ are incident.
The \df{$\Pi$-face-distance} $\dd_\Pi(v_1, v_2)$ of $v_1, v_2 \in V(G)$ is the minimum number $k$ such that there exists a sequence $u_0$, $f_0$, $u_1, \ldots, u_k,f_k$, $u_{k+1}$ such that $u_0 = v_1$, $u_{k+1} = v_2$,
and the face $f_i \in F(\Pi)$ is incident with $u_i$ and $u_{i+1}$, for $i = 0, \ldots, k$.
The \df{face-distance} $\dd_G(v_1, v_2)$ is the minimum $\Pi$-face-distance $\dd_\Pi(v_1, v_2)$ over all planar embeddings $\Pi$ of $G$. Note that the face-distance is 0 if and only if the graph $G+v_1v_2$ is planar.

The following result relating number of separating cycles and the face-distance of two vertices shall be used.

\begin{lemma}[Cabello and Mohar~\cite{cabello-2011}, Lemma~5.3]
\label{lm-separating-cycles}
Let $G$ be a planar graph and $x, y \in V(G)$.
Then the maximum number of disjoint $(x,y)$-separating cycles in $G$ is $\dd_G(x,y)$.
\end{lemma}

Let $C$ be a cycle in a $\Pi$-embedded graph $G$ and $\SS$ the surface where $G$ is 2-cell embedded by $\Pi$.
The cycle $C$ is \df{$\Pi$-contractible} if $C$ forms a surface-separating curve on $\SS$ such that one region
of $\SS - C$ is homeomorphic to an open disk.

Let $P_1, P_2, P_3$ be internally disjoint paths connecting vertices $u$ and $v$ in $G$.
If the cycles $P_1 \cup P_2$ and $P_2 \cup P_3$ are $\Pi$-contractible, then the cycle $P_1 \cup P_3$ is also $\Pi$-contractible (see~\cite{mohar-book}, Proposition 4.3.1).
This property is called \df{3-path-condition}.
Let $T$ be a spanning tree of $G$. A \df{fundamental cycle} of $T$ is the unique cycle in $T + e$ for an edge $e \in E(G) \sm E(T)$.

\begin{lemma}
  \label{lm-fund-cycles}
  Let $G$ be a $\Pi$-embedded graph, $L$ a K-graph in $G$, and $T$ a spanning tree of $L$.
  Then one of the fundamental cycles of $T$ in $L$ is $\Pi$-noncontractible.
\end{lemma}

\begin{proof}
  Suppose that all fundamental cycles of $T$ are $\Pi$-contractible.
  Since fundamental cycles of $T$ generate the cycle space of $L$, the 3-path-condition gives that each cycle of $L$ is $\Pi$-contractible.
  Thus $L$ separates the surface into three regions when $L\cong K_{3,3}$ and into four regions when $L\cong K_4$.
  Since $L$ is a K-graph in $G$, there is a principal $L$-bridge $B$ in $G$.
  But the attachments of $B$ does not lie on a single cycle of $L$ and thus $B$ cannot be embedded into any of the regions --- a contradiction.
\end{proof}

Since all cycles are contractible when genus is zero and any two $\Pi$-noncontractible cycles on the projective plane intersect, we have the following result.

\begin{lemma}
  \label{lm-noncontractible}
  Let $G$ be a $\Pi$-embedded graph. If $G$ contains two disjoint $\Pi$-noncontractible cycles,
  then $\eg(\Pi) \ge 2$.
\end{lemma}

The next lemma is a simple corollary of Lemmas~\ref{lm-fund-cycles} and~\ref{lm-noncontractible}.

\begin{lemma}
\label{lm-two-k-graphs}
  If $G$ satisfies one of the following conditions, then $\eg(G) \ge 2$.
  \begin{enumerate}[label=\rm(\roman*)]
  \item
    $G$ contains two disjoint K-graphs.
  \item
    $G$ contains a Kuratowski subgraph $K$ and a K-graph $L$ that intersects $K$ in at most one half-open branch of $K$.
  \item
    $G$ contains a Kuratowski subgraph $K$ and a K-graph $L$ homeomorphic to $K_{2,3}$
    such that $K$ and $L$ intersect in at most one branch $P$ of $K$,
    and the ends of $P$ do not lie on the same branch of $L$.
  \end{enumerate}
\end{lemma}

\begin{proof}
  If (i) holds, then the result follows by Lemmas~\ref{lm-fund-cycles} and~\ref{lm-noncontractible}.
  Suppose that (ii) holds and that $P$ is the branch of $K$ with ends $u$ and $v$ such that $V(L) \cap V(K) \ss V(P) \sm \{v\}$.
  The K-graph $L'$ in $G$ obtained from $K$ by deleting $u$ is disjoint from $L$.
  The result follows by~(i).

  Assume now that (iii) holds and that $P$ is the branch of $K$ with ends $u$ and $v$.
  Let $T$ be a spanning tree of $L$ such that $u$ and $v$ are its leaves.
  By Lemma~\ref{lm-fund-cycles}, there is a fundamental cycle $C$ of $T$ that is $\Pi$-noncontractible.
  Since $u$ and $v$ do not lie on a single branch of $L$ and they have degree 1 in $T$,
  we may assume that $C$ does not contain $u$. Thus, $K - u$ contains a K-graph disjoint from $C$.
  The result now follows by Lemmas~\ref{lm-fund-cycles} and~\ref{lm-noncontractible}.
\end{proof}

%-------------------------- Disjoint K-graphs --------------------------------------------------------------------
\section{Disjoint K-graphs in cascades}
\label{sc-k-graphs}

In this section, we show that for every cascade $G \in \S_1$, the graph $G^+$ contains two disjoint K-graphs.
We need the following property of separating cycles.

\begin{lemma}
\label{lm-close-cycles}
  Let $G$ be a planar graph, let $x,y \in V(G)$ be vertices separated by a cycle $C$, and let $H$ be the $x$-side of $C$.
  Then there exists an $(x,y)$-separating cycle $C'$ such that $C' \subseteq H \cup C$ and the $C'$-bridges
  containing $x$ and $y$ overlap.
\end{lemma}

\begin{proof}
  Pick $C'$ to be an $(x, y)$-separating cycle in $G$ such that $C' \subseteq H \cup C$ and that the distance in $O(G, C')$ of the $C'$-bridge $B_x$ containing $x$
  and the $C'$-bridge $B_y$ containing $y$ is minimum.
  Let $H'$ be the $x$-side of $C'$ and note that $H' \subseteq H$.

  Since $C'$ is $(x, y)$-separating, $B_x$ and $B_y$ have odd distance $d$ in $O(G, C')$.
  If $d = 1$, then $B_x$ and $B_y$ overlap. Hence we may assume that $d > 1$.
  Let $B_1, B_2$, and $B_3$ be the $C'$-bridges at distance 1, 2, and 3, respectively, from $B_x$ on a shortest path from $B_x$ to $B_y$ in $O(G, C')$.
  Since $B_2$ and $B_x$ do not overlap, the cycle $C'$ can be decomposed into two segments $Q_1$ and $Q_2$ with ends $v_1$ and $v_2$ such that
  $Q_1$ contains all attachments of $B_x$ and $Q_2$ contains all attachments of $B_2$.
  Furthermore, we can assume that $v_1$ and $v_2$ are attachments of $B_2$.
  Let $P$ be a path in $B_2$ connecting $v_1, v_2$ and let $C''$ be the cycle $Q_1 \cup P$.
  Let $B$ be a $C'$-bridge. If $B$ attaches to the interior of $Q_2$, then $B$ is a subgraph of a single $C''$-bridge $B_0$ containing $Q_2$.
  Note that this is the case for $B_1$ and $B_3$ since they $C'$-overlap with $B_2$.
  If $B$ does not attach to the interior of $Q_2$ it has the same attachments on $C''$ as on $C'$.
  Since $B_1$ only attaches to $Q_1$, we obtain that $B_1$ overlaps with $B_0$.
  It is not hard to see that $B_1$ and the $C''$-bridge containing $y$ have distance at most $d-2$ in $O(G, C'')$.
  Since $C'' \subseteq H' \cup C'$, we conclude that $C'' \subseteq H \cup C$.
  This contradicts the choice of $C'$.
\end{proof}

If $G \in \Gcxy$, then a \df{pre-K-graph} in $G$ is a subgraph of $G$ homeomorphic to either $K_4$ or $K_{2,3}$ that is a K-graph in $G\+$.
Separating cycles allow us to construct pre-K-graphs on each side of the cycle.

\begin{lemma}
\label{lm-x-k-graph}
  Let $C$ be an $(x, y)$-separating cycle in a planar graph $G \in \Gcxy$
  and let $B_x$ and $B_y$ be overlapping $C$-bridges containing $x$ and $y$, respectively.
  Then $G$ contains a pre-K-graph in $C \cup B_x$.
\end{lemma}

\begin{proof}
  Assume first that $B_x$ and $B_y$ skew-overlap and let $u_1, v_1$ be attachments of $B_x$
  and $u_2, v_2$ be attachments of $B_y$ such that $u_1, u_2, v_1, v_2$ appear on $C$ in this order.
  Let $P$ be a path connecting $u_1$ and $v_1$ in $B_x$. We see that $P \cup C$ is a pre-K-graph in $G$.

  Assume now that $B_x$ and $B_y$ do not skew-overlap. Hence $B_x$ and $B_y$ have three attachments $u_1, u_2, u_3$ in common.
  Let $P_1, P_2, P_3$ be internally disjoint paths in $B_x$ with one common end $u$ and with the other ends being $u_1, u_2, u_3$, respectively.
  Let $P$ be a (possibly trivial) path connecting $x$ and $P_1 \cup P_2 \cup P_3$ in $B_x - C$ and let $v$ be the other end of $P$.
  If $v = u$, then $C \cup P_1 \cup P_2 \cup P_3$ is a pre-K-graph in $G$.
  If $v \in V(P_1)\setminus \{u\}$, then let $C'$ be the segment of $C$ with ends $u_2$ and $u_3$ that contains $u_1$.
  We have that $C' \cup P_1 \cup P_2 \cup P_3$ is a pre-K-graph in $G$ homeomorphic to $K_{2,3}$ with branch vertices $u$ and $u_1$.
  We construct a pre-K-graph similarly if $v \in (V(P_2) \cup V(P_3))\setminus \{u\}$.
\end{proof}

We have the following corollary.

\begin{corollary}
\label{cr-two-cycles}
  Let $G$ be a planar graph in $\Gcxy$. If $\dd_G(x,y) \ge 2$,
  then $G$ contains two disjoint pre-K-graphs.
%%   then there exists disjoint cycles $C_1$ and $C_2$ in $G$ such that the $C_1$-bridges and $C_2$-bridges
%%   containing $x$ and $y$ overlap.
\end{corollary}

\begin{proof}
  By Lemma~\ref{lm-separating-cycles}, there are two disjoint $(x,y)$-separating cycles $C_1$ and $C_2$ in $G$.
  Let $C_1$ and $C_2$ be such that the $x$-side of $C_1$ and the $y$-side of $C_2$ are disjoint.
  By Lemma~\ref{lm-close-cycles}, there is an $(x,y)$-separating cycle $C_1'$ such that the $C_1'$-bridges containing $x$ and $y$ overlap.
  Similarly, there is an $(x,y)$-separating cycle $C_2'$ such that the $C_2'$-bridges containing $x$ and $y$ overlap.
  Furthermore, we can pick $C_1'$ and $C_2'$ so that $C_1'$ is contained in the $x$-side of $C_1$ and $C_2'$ in the $y$-side of $C_2$.
  Therefore, $C_1'$ and $C_2'$ are disjoint. Let $B_x$ be the $C_1'$-bridge containing $x$ and let $B_y$ be the $C_2'$-bridge containing $y$.
  By Lemma~\ref{lm-x-k-graph}, the graph $G$ contains a pre-K-graph in $C_1' \cup B_x$ and a pre-K-graph in $C_2' \cup B_y$.
  Thus, $G$ contains two disjoint pre-K-graphs.
\end{proof}

The following lemma relates the face-distance of $x$ and $y$ in a planar graph $G \in \Gcxy$ to the genus of $G\+$.

\begin{lemma}
  \label{lm-dd-xy-2}
  Let $G \in \Gcxy$ be a planar graph and $\dd = \dd_G(x,y)$. If $\dd\le2$, then $\egp(G)=\dd$. If $\dd\ge3$, then $\egp(G)=2$.
\end{lemma}

\begin{proof}
  Suppose first that there exists a planar embedding $\Pi$ of $G$ where $\dd_\Pi(x, y) \le 1$.
  If $\dd_\Pi(x, y) = 0$, then $G\+$ is planar and $\egp(G) = 0$.
  Suppose then that $\dd_\Pi(x, y) = 1$. Then there exists a vertex $v \in V(G)$ and two $\Pi$-faces $f_1, f_2$ incident with $v$
  such that $f_1$ is incident with $x$ and $f_2$ is incident with $y$.
  Let $e_1ve_2$ be a $\Pi$-angle of $f_1$ and $e_3ve_4$ a $\Pi$-angle of $f_2$.
  We can write the local rotation around $v$ as $e_2,S_1,e_3,e_4,S_2,e_1$.
  Let us construct the following embedding $\Pi'$ of $G\+$ in the projective plane.
  Let $\Pi'(u) = \Pi(u)$ for each $u \in V(G) \sm \{x, y, v\}$.
  To obtain $\Pi'(x)$, insert the edge $xy$ into the local rotation $\Pi(x)$ of $x$ between the edges $e_1', e_2'$ where $e_1', x, e_2'$ is a $\Pi$-angle of $f_1$.
  The local rotation $\Pi'(y)$ of $y$ is obtained analogously.
  Let $\Pi'(v) = e_2, S_2,e_3,e_1,S_2^{\rm R},e_4$, where $S_2^{\rm R}$ is the reverse of $S_2$.
  Let $\Pi'(e) = -1$, if $e \in \{xy, e_1, e_4\} \cup S_2$, and $\Pi'(e) = 1$ otherwise.
  We leave it to the reader to check that $\Pi'$ is indeed an embedding of $G\+$ into the projective plane.
  Thus $\egp(G) \le 1$ as claimed.

  Assume now that $\dd_G(x,y) \ge 2$.
  By Corollary~\ref{cr-two-cycles}, $G\+$ contains two disjoint K-graphs.
  By Lemma~\ref{lm-two-k-graphs}(i), $\eg(G\+) = \egp(G) \ge 2$.
  However, adding an edge increases Euler genus by at most 2, so $\egp(G)=2$.
\end{proof}

A pre-K-graph $L$ in a planar graph $G \in \Gcxy$ is a \df{$z$-K-graph} for a terminal $z \in \{x, y\}$
if $z \in V(L)$ and, if $L\cong K_4$, then $z$ is a branch vertex of $L$,
and, if $L\cong K_{2,3}$, then $z$ lies on an open branch of $L$.
The \df{boundary} of $L$ is the cycle of $L$ that consists of all branches of $L$ that are not incident with $z$.
All vertices and edges of $L$ that do not lie on the boundary of $L$ are said to be in the \df{interior} of $L$.
A graph $G \in \Gcxy$ contains \df{disjoint $xy$-K-graphs} if it contains
an $x$-K-graph and a $y$-K-graph that are disjoint.
We conclude this section by showing that each graph in $\S_1$ contains disjoint $xy$-K-graphs.

\begin{lemma}
  \label{lm-disjoint-xy-k-graphs}
  Each graph in $\S_1$ contains disjoint $xy$-K-graphs.
\end{lemma}

\begin{proof}
  Let $G \in \S_1$.
  By (C1) and (C3) from the definition of cascades, there is a minor-operation $\mu \in M(G)$ such that $\mu G$ is planar but $\egp(\mu G) = 2$.
  By Lemma~\ref{lm-dd-xy-2}, $\dd_{\mu G}(x,y) \ge 2$.
  Minor operations cannot increase the face-distance. Thus, $\dd_G(x,y)\ge2$.
  By Corollary~\ref{cr-two-cycles} and its proof, $G$ contains disjoint $(x,y)$-separating cycles $C_1', C_2'$ and disjoint pre-K-graphs $L_x\subseteq C_1'\cup B_x$ and $L_y\subseteq C_2'\cup B_y$ (where $B_x$ is the $C_1'$-bridge containing $x$ and $B_y$ is the $C_2'$-bridge containing $y$ such that $B_x \cap B_y = \emptyset$).

  Suppose that $x \not\in L_x$. Then $x$ has a neighbor $v\in V(B_x)\setminus V(C_1')$.
  Consider contracting the edge $xv$.
  Since $L_x$ and $L_y$ are disjoint pre-K-graphs in $G/xv$, we have that $\egp(G/xv) \ge 2$.
  By (C1), $G/xv$ is planar.
  Since $B_x /xv$ has the same attachments on $C_1'$ as $B_x$ and $G$ is nonplanar,
  we conclude that $C_1' \cup B_x$ is nonplanar and thus contains a Kuratowski subgraph $K$.
  Let $e$ be an edge of $G$ joining a vertex on $C_1'$ with a vertex that is not in $C_1'\cup B_x$. Observe that $L_y$ is a pre-K-graph in $G/e$.
  In the graph $G / e$,  $K$ shares at most one vertex with $L_y$.
  By Lemma~\ref{lm-two-k-graphs}(ii), $\egp(G/e) \ge 2$.
  Since $G/e$ contains $K$, $\eg(G/e) \ge 1$, a contradiction with (C1).
  We conclude that $x \in L_x$.
  By symmetry $y \in L_y$.
  Therefore, $G$ contains disjoint $xy$-K-graphs.
\end{proof}

%------------------------ Cascades for the Klein Bottle ---------------------------------------------------

\section{The class $\S_1$}
\label{sc-s1}

Throughout this section we will use the following notation and assumptions.
Let us consider a graph $G \in \S_1$.
By Lemma~\ref{lm-disjoint-xy-k-graphs}, $G$ contains an $x$-K-graph $L_x$ and a $y$-K-graph $L_y$ that are disjoint.
We shall assume that $L_x$ is \df{minimal} in the sense that there is no $x$-K-graph properly contained in $L_x$. Similarly take $L_y$ minimal.
Let $\By$ be the $L_x$-bridge in $G$ that contains $L_y$. Define $\Bx$ similarly.
A \df{base} in $G$ is a subgraph $H$ of $G$ such that $H$ contains $L_x$ and $L_y$ and they are pre-K-graphs in $H$.
In this section, we use the structure obtained in the previous section to construct cascades in $\S_1$
and find their planar bases.

Each graph $G \in \S_1$ has $\egp(G) = 2$, and thus contains a graph $H \in \Cc_1(\egp)$ as a minor. The next lemma shows that $H$ has to be planar.

\begin{lemma}
  \label{lm-c2-egp}
  Let $G \in \Gcxy$ be a cascade in $\S_1$.

  {\rm (a)} If $H\in \Gcxy$ is a proper minor of $G$ and $\egp(H)=2$, then $H$ is a planar graph.

  {\rm (b)} If $H \in \Gcxy$ is a minor of $G$ such that $H \in \Cc_1(\egp)$, then $H$ is planar.
\end{lemma}

\begin{proof}
  In case (b), $H$ is a proper minor of $G$ as wellby the properties (C1) and (C3) of cascades. Thus, in both cases, (a) and (b),
  there exists a minor operation $\mu \in \M(G)$ such that $H$ is a minor of $\mu G$.
  Since $\egp(H) = 2$ and $\egp$ is minor-monotone, $\egp(\mu G) \ge \egp(H) = 2$.
  By (C1), $\eg(H) < g(H) = 1$, which means that $H$ is planar.
\end{proof}

\begin{figure}
  \centering
  \includegraphics{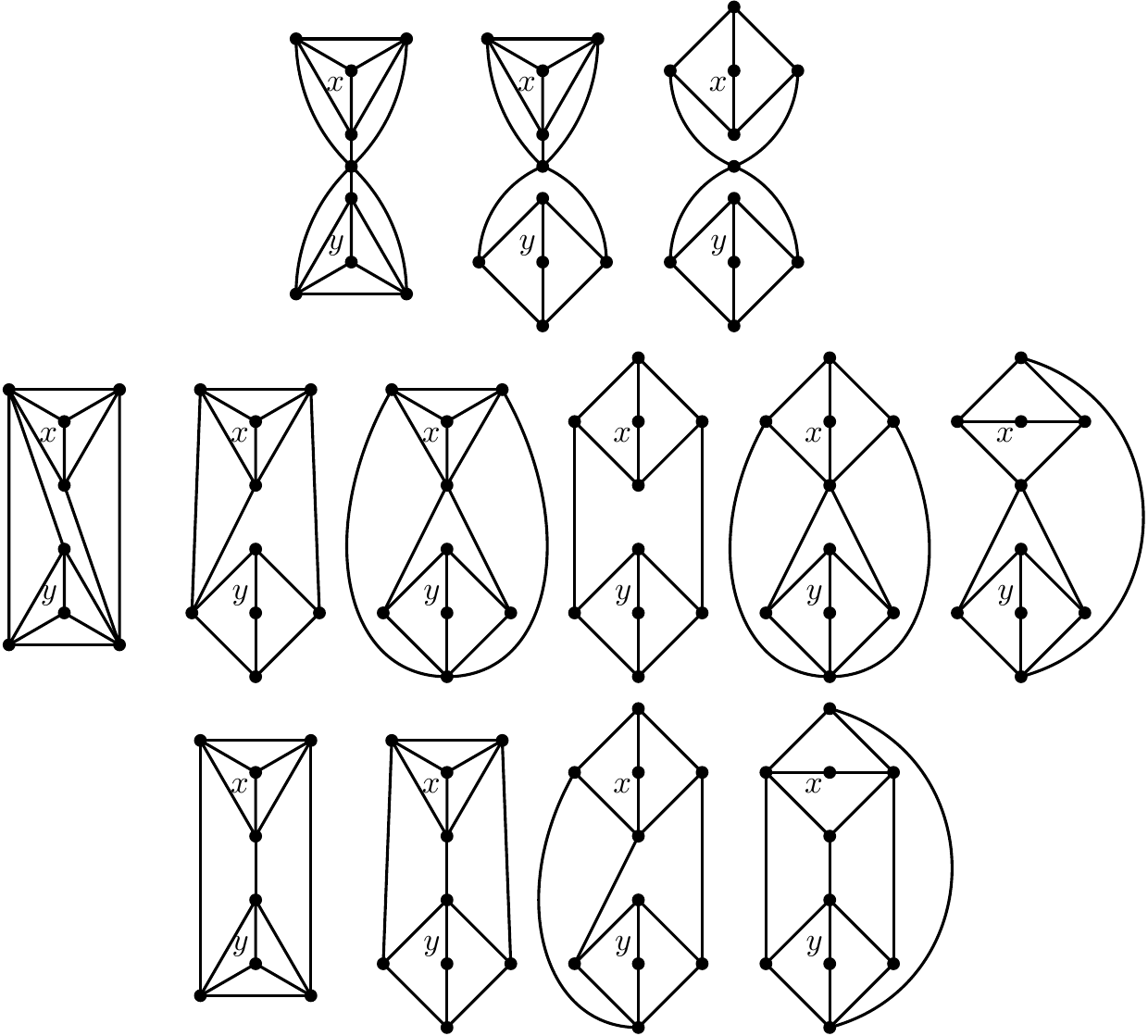}
  \caption{The planar graphs in $\Cc_1(\egp)$.}
  \label{fg-projective}
\end{figure}

Lemma~\ref{lm-c2-egp} combined with (C2) implies that each graph $G \in \S_1$ contains a planar graph $H\in \Cc_1(\egp)$ as a minor. By Lemma \ref{lm-cc-egp-iff}, $\widehat{H}^+$ is either in $\E_1$ (an obstruction for the projective plane) or $\eg(H/xy)=\egp(H)=2$. In the latter case, Lemma \ref{lm-cc-egp} shows that $\widehat{H}^+ \in \E_1^*$ and $\widehat{H/xy} \in \E_1$.
The complete list of planar graphs in $\Cc_1(\egp)$ is depicted in Fig.~\ref{fg-projective}.
The list has been obtained as follows: We start with $\E_1$ which consists of 35 obstructions for the projective plane \cite{archdeacon-1981,glover-1979}. Every planar graph obtained from one of these by removing an edge and using its ends as terminals $x$ and $y$ gives one of the graphs. Next, each of 68 ($=103-35$) graphs $Q\in \E_1^*\setminus \E_1$ (cf. \cite{glover-1979}) is tested to check if the removal of an edge $xy$ yields a planar graph $H\in\Gcxy$ such that $H/xy = Q/xy \in \E_1$.
A simple use of computer then reveals that the resulting planar cases are precisely those depicted in Fig.~\ref{fg-projective}.

\begin{theorem}
\label{thm:C1plus}
The class $\Cc_1(\egp)$ contains precisely 13 planar graphs that are depicted in Fig.~\ref{fg-projective}.
Every cascade in $\S_1$ contains one of these as a minor.
\end{theorem}

\begin{corollary}
\label{cor:C1plusplanar}
Every planar graph in $\Cc_1(\egp)$ contains disjoint $xy$-K-graphs.
\end{corollary}

The corollary can be proved by inspection of graphs in Fig.~\ref{fg-projective}. However,
it is not hard to see that Lemma~\ref{lm-disjoint-xy-k-graphs} can be adapted to prove the corollary directly,
without relying on the computer-assisted proof of Theorem \ref{thm:C1plus}.

A selection of nonplanar graphs in $\Cc_1(\egp)$ is depicted in Fig.~\ref{fg-selected}.
The consequence of Lemma~\ref{lm-c2-egp}(b) is that a graph $G \in \S_1$ cannot contain a graph in Fig.~\ref{fg-selected} as a minor. This will be used extensively in the proofs of Lemmas \ref{lm-sep-1} and \ref{lm-2-sep}.

Next we prove, using minimality assumption on $L_x$, that in the case when $L_x\cong K_4$, $\By$ is attached to $L_x$ only at the branch-vertices of $L_x$.

\begin{lemma}
\label{lm-k4}
  If $L_x\cong K_4$, then the attachments of $\By$ in $L_x$ are branch-vertices of $L_x$.
\end{lemma}

\begin{proof}
  Let $w_0=x, w_1, w_2$, and $w_3$ be the branch-vertices of $L_x$ and let $P_{i,j}$ be the open branch of $L_x$ connecting $w_i$ and $w_j$.
  Assume for a contradiction that there is an attachment $w$ of $\By$ on an open branch of $L_x$.
  Suppose first that $w$ lies on $P_{1,2}$.
  Then there is an $x$-K-graph $L\cong K_{2,3}$ and disjoint from $L_y$: The subgraph $L$ consists of the branch vertices $w_1, w_2$
  and branches $P_{1,3} \cup P_{3,2}$, $P_{1,2}$, and $P_{1,0} \cup P_{0, 2}$. Since $\By$ attaches to vertices $x, w_3$, and $w$, and $L$ is a proper subgraph of $L_x$,
  $L$ is indeed an $x$-K-graph disjoint from $L_y$, a contradiction to the minimality of $L_x$.

  By symmetry, we may assume that $w$ lies on $P_{1,0}$.
  Let $e$ be the edge of $L_x$ incident with $x$ and $P_{1,0}$.
  Consider the graph $G' = G / e$.
  Since $L_x / e$ is an $x$-K-graph of $G$ disjoint from $L_y$, $G'\+$ contains two disjoint K-graphs and thus $\eg(G'\+) = 2$ by Lemma~\ref{lm-two-k-graphs}(i).
  If $e = wx$, then $L_x / e$ is a K-graph in $G'$ and $\eg(G') \ge 1$.
  Otherwise, $L_x/e$ contains a K-graph $L\cong K_{2,3}$ as follows.
  The branch-vertices of $L$ are $w_1$ and $x$. The branches of $L$ are paths $P_{1,2} \cup P_{2,0}$, $P_{1,3} \cup P_{3,0}$, and $P_{1,0}$.
  Since $\By$ attaches on to vertices $w_2, w_3$, and $w$, the subgraph $L$ is a K-graph in $G'$ and $\eg(G') \ge 1$.
  We conclude that $\eg(G') = \eg(G)$ and $\eg(G'\+) = \eg(G\+)$ which violates (C1).
\end{proof}

Since $\eg(G) = 1$, at most one of $L_x$ and $L_y$ can be a K-graph in $G$ (Lemma \ref{lm-noncontractible}).
Let us recall that the interior of $L_x$ consists of $x$ and all open branches of $L_x$ that are incident with~$x$.

\begin{lemma}
\label{lm-interiors}
  If $B_y$ is attached to the interior of $L_x$, then its only attachment in the interior of $L_x$ is the vertex $x$. In such a case, $\Bx$ is not attached to the interior of $L_y$.
\end{lemma}

\begin{proof}
  If both $\By$ and $\Bx$ attach to the interior of $L_x$ and $L_y$, respectively, then we obtain (using Lemma~\ref{lm-k4} if $L_x$ or $L_y$ is homeomorphic to $K_4$) that both $L_x$ and $L_y$ are K-graphs in $G$.
  By Lemma~\ref{lm-two-k-graphs}(i), $\eg(G) \ge 2$, a contradiction with $G \in \S_1$.

  Suppose that $\By$ has an attachment in the interior of $L_x$ that is different from $x$.
  Thus there exists an edge $e \in E(L_x)$ with both ends in the interior of $L_x$.
  Consider the graph $G / e$. Since $L_x/e$ is a K-graph in $G / e$, $\eg(G / e) \ge 1$.
  Since $L_x/e$ and $L_y$ are $xy$-K-graphs in $G / e$, $\egp(G / e) \ge 2$.
  This contradicts (C1).
\end{proof}

When dealing with cascades in $\S_1$, we will consider a base $H$ in $G$ containing the $xy$-K-graphs $L_x$ and $L_y$ in $G$ as introduced at the beginning of this section. We will explore how $L_x$ and $L_y$ are linked to each other by paths in $H$. To describe the linkages, we introduce some additional terminology that will be used to capture the situation inside the graph $H$.

Let $H$ be a graph that contains a subgraph $L$, called \df{core}, homeomorphic to $K_4$ or $K_{2,3}$ with distinguished cycle $C$ in $L$ that contains two or three branch vertices of $L$.
When $L=L_x$ or $L=L_y$, then we select $C$ to be the cycle that does not contain the terminal $x$ or $y$, respectively.
We say that $C$ is a \df{boundary cycle} of the core $L$.
The edges and vertices of $L$ that do not lie in $C$ are said to be in the \df{interior} of $L$.
For $U \ss V(H)$, we say that $L$ is \df{$U$-linked} in $H$ if there are $|U|$ disjoint paths in $H$ connecting $C$ and $U$ that are internally disjoint from $L$. We say that $H$ is a \df{$U$-linkage} of $L$ if $L$ is $U$-linked in $H$ and
the following holds. If $L\cong K_{2,3}$, then for every open branch $t$ on the boundary of $L$ there is a path in $H$ from $t$ to $U$ that is internally disjoint from $L$; if $L\cong K_4$, then for every branch vertex $t$ on the boundary of $L$ there is a path in $H$ from $t$ to $U$ that is internally disjoint from $L$. Existence of these paths will enable us to show that $L$ is close to be a K-graph in $H$. Namely, if we add a new vertex adjacent to all vertices in $U$ and to a vertex in the interior of $L$, then $L$ contains a K-graph in the extended graph.
If $u \in U$ has degree at least 2 in a $U$-linkage $H$, then $u$ is called a \df{foot} of $H$.
If $u \in U$ has degree 1, then the \df{foot} of $H$ containing $u$ is the path from $u$ to a first vertex of degree at least 3.
The foot containing $u$ is also called the \df{$u$-foot} of $H$.
A \df{$u$-foot} is \df{removable} if $H$ is a $(U \sm \{u\})$-linkage.
The notion of a linkage will be used to describe a pre-K-graph in $G$ together with essential paths that attach onto it.

A set $U \ss V(G)$ \df{separates} $L_x$ and $L_y$ in $G$ if every $(L_x,L_y)$-path in $G$ contains a vertex in $U$.
We say that $U$ \df{blocks} $L_x$ \df{from} $L_y$ in $G$ if $U \cup \{x\}$ separates $L_x$ and $L_y$ and $L_x$ is $U$-linked in $G$. The introduced terms are illustrated in Fig.~\ref{fg-linkage-example}.

\begin{figure}
  \centering
  \includegraphics{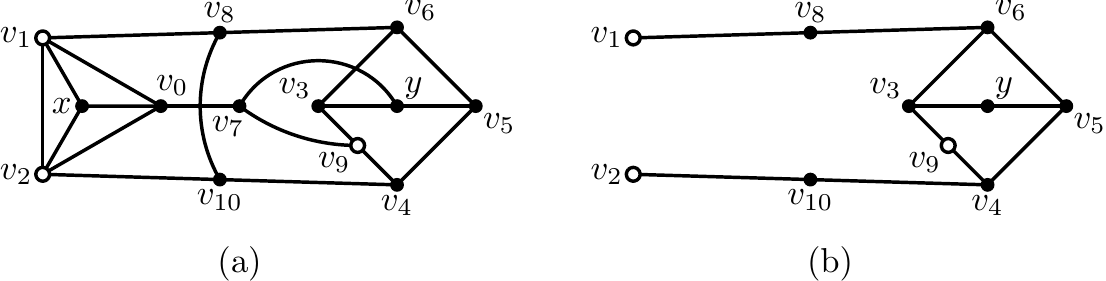}
  \caption{%
    (a) A graph with $L_x$ induced by $x, v_0, v_1, v_2$ and $L_y$ induced by $y, v_3, v_4, v_5, v_6, v_9$.
    The set $U = \{v_1, v_2, v_9\}$ blocks $L_y$ from $L_x$ but not $L_x$ from $L_y$.
    (b) A $U$-linkage with core $L_y$ and feet $v_4v_{10}v_2$, $v_9$, $v_6v_8v_1$.
    Each of the feet $v_4v_{10}v_2$ and $v_9$ is removable but $v_6v_8v_1$ is not.
    This shows that $L_y$ and $U$ admit the linkage \gs{g} (shown in Fig.~\ref{fg-separation}(g)) that is obtained by contracting the edges $v_3v_9, v_1v_8, v_2v_{10}$.}
  \label{fg-linkage-example}
\end{figure}

Let $k$ be the maximum number of pairwise disjoint paths in $G$ connecting
the boundaries of $L_x$ and $L_y$ that are internally disjoint from $L_x$ and $L_y$.
Then we say that the $xy$-K-graphs $L_x$ and $L_y$ are \df{$k$-separated} in $G$.

\begin{lemma}
\label{lm-separating}
  If $L_x$ and $L_y$ are $k$-separated, then there exists a set $U \ss V(G)$ of cardinality $k$ such that
  one of the following cases occurs:
  \begin{enumerate}[label=\rm(\roman*)]
  \item
    $U$ blocks $L_x$ from $L_y$ and $L_y$ from $L_x$.
  \item
    $U$ blocks $L_x$ from $L_y$ and $U \cup \{x\}$ blocks $L_y$ from $L_x$.
  \item
    $U \cup \{y\}$ blocks $L_x$ from $L_y$ and $U$ blocks $L_y$ from $L_x$.
  \end{enumerate}
\end{lemma}

\begin{proof}
  In the conclusions of the lemma, there is symmetry between $x$ and $y$. Thus, we may assume by Lemma~\ref{lm-interiors} that $\By$ is not attached to the interior of $L_x$.
  Let $P_1, \ldots, P_r$ be pairwise disjoint paths connecting $L_x$ and $L_y$ such that $r$ is maximum
  and let $U_0$ be a minimum vertex-set that meets all paths connecting $L_x$ and $L_y$.
  By Menger's Theorem, we have that $|U_0| = r$.
  Note that $r \ge k$. Assume first that $r = k$.
  In this case $U_0$ separates $L_x$ and $L_y$.
  Since there are $k$ pairwise disjoint paths connecting the boundaries of $L_x$ and $L_y$ and all of them
  meet $U_0$, both $L_x$ and $L_y$ are $U_0$-linked in $G$. We conclude that (i) holds.

  Assume now that $r > k$. By Lemma~\ref{lm-interiors}, $\Bx$ has at most one attachment in the interior of $L_y$.
  Thus there is only one path, say $P_r$, that has an end in the interior of $L_y$.
  As noted at the beginning of the proof, none of the paths is attached to the interior of $L_x$.
  Since there are at most $k$ disjoint paths joining the boundaries of $L_x$ and $L_y$,
  we conclude that $r = k+1$.
  Let $U_1$ be a minimum vertex-cut (of size $k$) that meets all paths connecting the boundaries of $L_x$ and $L_y$.
  Thus $U_1$ meets all the paths $P_1, \ldots, P_k$.
  We see that $U_1 \cup \{y\}$ separates $L_x$ and $L_y$.
  Also, the paths $P_1, \ldots, P_r$ demonstrate that $L_x$ is $(U_1 \cup \{y\})$-linked
  and that $L_y$ is $U_1$-linked.
  We conclude that $U_1 \cup \{y\}$ blocks $L_x$ from $L_y$ and $U_1$ blocks $L_y$ from $L_x$.
  Hence (iii) holds. The case (ii) occurs in the symmetric case when $\By$ attaches to the interior of $L_x$.
\end{proof}

\begin{figure}
  \centering
  \includegraphics{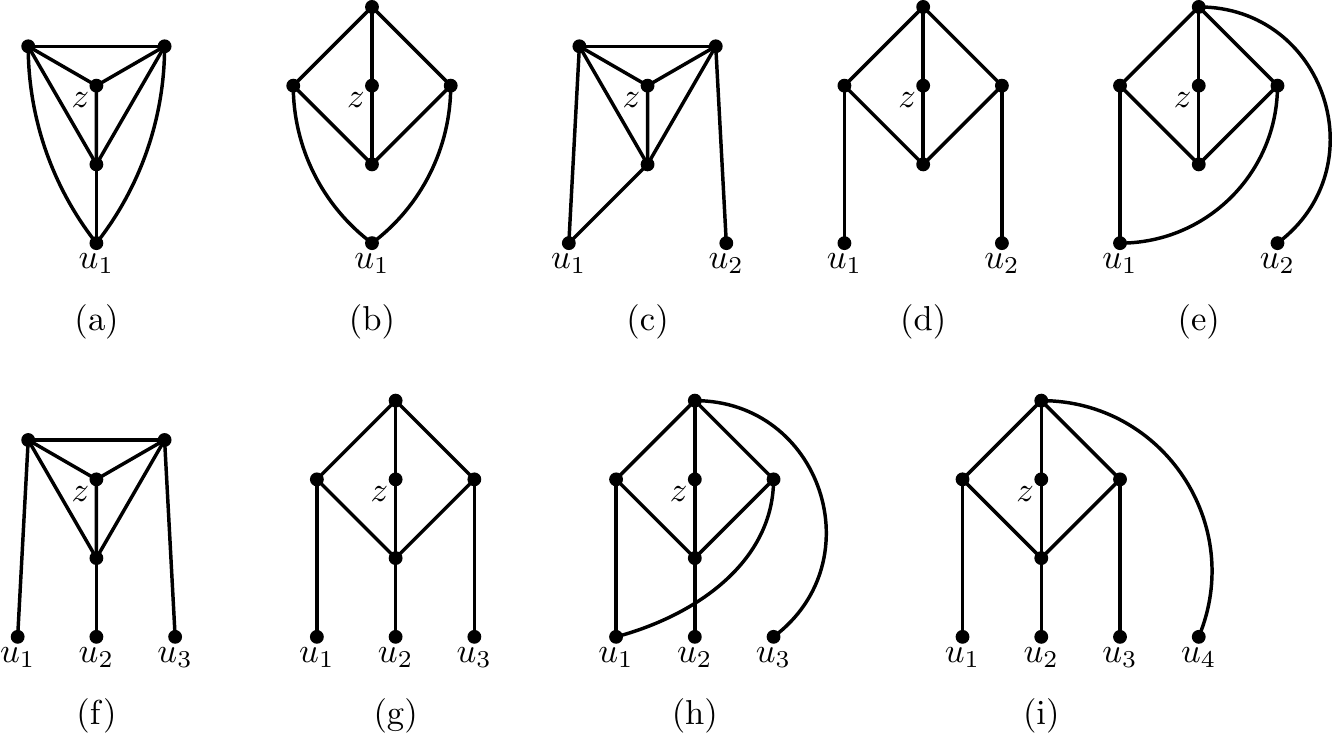}
  \caption{Linkages to small sets. In each linkage, any subset of feet can be contracted.}
  \label{fg-separation}
\end{figure}

In the next lemma we classify all possible types of $U$-linkages of small order. To do this, we need a way to say when an abstract $U'$-linkage $H$ models a $U$-linkage in $G$.
Consider the cascade $G$ and let $z \in \{x,y\}$ and $U \ss V(G)$.
We say that $L_z$ and $U$ \df{admit} a $U'$-linkage $H$ if there exists a set $F \ss E(G)$ such that $L_z /(F\cap E(L_z))$ is a $z$-K-graph in $G/F$ and
$H$ is isomorphic to a subgraph of $G/F$ such that $U'$ is mapped bijectively to $U$ and the core of $H$ is mapped to $L_z/(F\cap E(L_z))$.

\begin{lemma}
  \label{lm-sep-cuts}
  Let $H$ be a base of $G$ and let $U \ss V(H)$.
  If $U$ blocks $L_x$ from $L_y$ in $H$, then $L_x$ and $U$ admit a linkage.
  Furthermore, if $1 \le |U| \le 4$, then $L_x$ and $U$ admit a linkage from Fig.~\ref{fg-separation} (with some of the feet possibly of length zero).
\end{lemma}

\begin{proof}
  Since $H$ is a base, it contains $L_x$ and $L_y$, and these are K-graphs in $H^+$.
  Since $U$ blocks $L_x$ from $L_y$ in $H$, there are three paths $P_1, P_2, P_3$ joining the branch vertices on the boundary of $L_x$ with $U$ (when $L_x\cong K_4$) or two paths $P_1, P_2$ from the interiors of both open branches on the boundary of $L_x$ to $U$ (when $L_x\cong K_{2,3}$). These paths are internally disjoint from $L_x$ by Lemma \ref{lm-interiors}. Moreover, $L_x$ is $U$-linked in $H$,
  so there are $|U|$ disjoint paths $Q_1,\dots,Q_{|U|}$ joining the boundary of $L_x$ with $U$.
  By definition, the union $R = L_x \cup_i P_i \cup_i Q_i$ form a $U$-linkage of $L_x$ in $H$.

  Let us now prove that $L_x$ and $U$ admit a linkage from Fig.~\ref{fg-separation} when $|U| \le 4$.
  Assume first that $L_x\cong K_4$.
  By Lemma~\ref{lm-k4}, $|U| \le 3$ and
  there are three paths $P_1, P_2, P_3$ connecting the branch-vertices of $L_x$ different from $x$ to $U$.
  Choose the paths so that each pair is disjoint if possible.
  Assume that $U = \{u_1\}$.
  By contracting the edges of $P_1, P_2$, and $P_3$ that are not incident with $L_x$,
  we obtain that $L_x$ admits the linkage \gs{a}.

  Assume now that $U = \{u_1, u_2\}$ is of size two.
  Since the paths $Q_1,Q_2$ also start at the branch vertices of $L_x$ (by Lemma \ref{lm-k4}), we may assume that $P_1$ and $P_2$ are disjoint and connect $L_x$ to $u_1$ and $u_2$, respectively.
  We may also assume that $P_3$ intersects only one of the other paths, say $P_1$.
  By contracting the edges of $P_1, P_2,P_3$ that are not incident with $L_x$, we obtain that $L_x$ admits the linkage \gs{c}.

  Assume now that $U = \{u_1, u_2, u_3\}$ is of size three.
  Since there are three disjoint paths $Q_1,Q_2,Q_3$ connecting $L_x$ and $U$, we may assume that $P_1, P_2$, and $P_3$ are pairwise disjoint.
  Thus $L_x$ admits the linkage \gs{f}.

  Assume now that $L_x\cong K_{2,3}$.
  There are two paths $P_1, P_2$ connecting the open branches on the boundary of $L_x$ to $U$.
  Choose the paths so that they are disjoint if possible.
  Assume that $U = \{u_1\}$. We see that $L_x$ admits the linkage \gs{b}.
  Assume now that $U = \{u_1, u_2\}$ is of size two.
  After possibly changing some of the paths, we may assume that $Q_1 = P_1$. If $P_2$ is disjoint from $P_1$, then $L_x$ admits the linkage \gs{d}.
  Otherwise we may assume that $Q_2$ is disjoint from $P_2$ and from the open branches of $L_x$.
  Hence $L_x$ admits the linkage \gs{e}.

  Assume now that $U = \{u_1, u_2, u_3\}$ is of size three.
  We may assume that $Q_1 = P_1$. If $P_2$ is disjoint from $P_1$, then $P_2$ can be changed, if necessary, so that it intersects only one of $Q_2,Q_3$. Then it is easy to see that $L_x$ admits the linkage \gs{g}.
  Otherwise, we may assume that $P_2$ intersects $P_1$ and that its segment from $L_x$ to $P_1$ does not intersect $Q_1,Q_2$. Now it is easy to see that $L_x$ admits the linkage \gs{h}.

  Assume now that $U = \{u_1,u_2,u_3,u_4\}$ is of size four.
  We may assume that $Q_1 = P_1$. If $P_2$ first intersects one of $Q_2, Q_3, Q_4$, then $L_x$ admits the linkage \gs{i}.
  If $P_2$ first intersects $P_1$, then one of $Q_2, Q_3, Q_4$ connects to an open branch of $L_x$ which is a contradiction with the choice of $P_1,P_2$, since $Q_2, Q_3$, and $Q_4$ are disjoint from $P_1$.
\end{proof}

The following lemma will be used to reduce the number of cases when $G$ admits linkages for $L_x$ and $L_y$ whose feet meet each other.

\begin{lemma}
\label{lm-removable}
Suppose that $H$ is a base of $G$ such that $L_x$ admits a $U_1$-linkage $H_x$ and $L_y$ admits a $U_2$-linkage $H_y$ in $H$
such that $U_1 \sm U_2 \ss \{y\}$, $U_2 \sm U_1 \ss \{x\}$,
$H_x$ and $H_y$ are edge-disjoint,
and there exists $u \in U_1 \cap U_2$ such that the $u$-feet of $H_x$ and $H_y$ are removable.
Then there is a proper subbase of $H$.
Moreover, neither $L_x$ nor $L_y$ is a K-graph in $G$.
\end{lemma}

\begin{proof}
  Since $L_x$ is $U_1$-linked, there are pairwise-disjoint paths $P_v$, $v \in U_1$, connecting $L_x$ and $U_1$.
  Similarly, there are pairwise-disjoint paths $Q_v$, $v \in U_2$ connecting $L_y$ and $U_2$.
  We may assume by symmetry that $u \not\in V(L_x)$. Thus $P_u$ is a non-trivial path (but $Q_u$ may possibly consist of a single vertex, $u$).

  Let $v_1v_2$ be the edge in $P_u$ such that $v_1 \in V(L_x)$ and $H' = H - v_1v_2$. We claim that $H'$ is a base in $G$.
  Since $u$ is a removable foot of $H_x$ and $H_y$ and $v_1v_2 \not\in E(H_y)$, $H_x$ is $(U_1 \sm \{u\})$-linkage of $L_x$ in $H'$ and
  $H_y$ is a $(U_2 \sm \{u\})$-linkage of $L_y$ in $H'$.
  Since, for each $v \in U_1 \sm \{x, u\}$, it holds that $v \in U_2$, there is a path in $H_y$ connecting $v$ and $y$.
  Thus $L_x$ is a pre-K-graph in $H'$. Similarly, $L_y$ is also a pre-K-graph in $H'$.
  We conclude that $H'$ is a base of $G$.
  This proves the first part of the lemma.

  To prove the remaining claim, suppose for a contradiction that $L_x$ is a K-graph in $G$. Let $G' = G - v_1v_2$.
  Since $H'$ is a base of $G'$, we have that $\egp(G') \ge 2$.
  Since $v_1 \in V(L_x)$, the $L_x$-bridge in $G'$ containing $y$ attaches to the same vertices of $L_x$ as the $L_x$-bridge in $G$ containing $y$, except possibly to $v_1$.
  Therefore, $L_x$ is a K-graph in $G'$ as $u$ is a removable foot of $H_x$.
  Thus $\eg(G') \ge 1$ which contradicts (C1).
  The case when $L_y$ is a K-graph in $G$ is done similarly.
\end{proof}

Suppose that $L_x$ and $L_y$ are $k$-separated.
By Lemma~\ref{lm-separating}, there exists a set $U$ of size $k$ such that a statement (i), (ii), or (iii) of that lemma holds.
If (i) holds, then $L_x$ and $L_y$ are blocked from each other by $U$.
Otherwise, we may assume that (ii) holds and $L_x$ is blocked from $L_y$ by $U$ and $L_y$ is blocked from $L_x$ by $U \cup \{x\}$.
By Lemma~\ref{lm-sep-cuts}, $L_x$ admits a $U$-linkage $H_x$ and $L_y$ admits a $U_y$-linkage $H_y$.
Assume that $H_x$ and $H_y$ are minimal (with respect to taking subgraphs).

If $|U| \le 4$, then Lemma~\ref{lm-sep-cuts} asserts that $H_x$ is one of the linkages in Fig.~\ref{fg-separation}.
In that case, let $u_1, \ldots, u_k$ be the vertices of $U$ to which $H_x$ is linked as depicted in Fig.~\ref{fg-separation}.
Similarly, when $|U_y| \le 4$, $H_y$ is one of the linkages in Fig.~\ref{fg-separation}.
Let $u_1', \ldots, u_r'$ be the vertices of $U_y$ in the order in which they are depicted in the picture of $H_y$ in Fig.~\ref{fg-separation}.
In the following series of lemmas we shall describe all cascades that are at most 2-separated.

\begin{figure}
  \centering
  \includegraphics{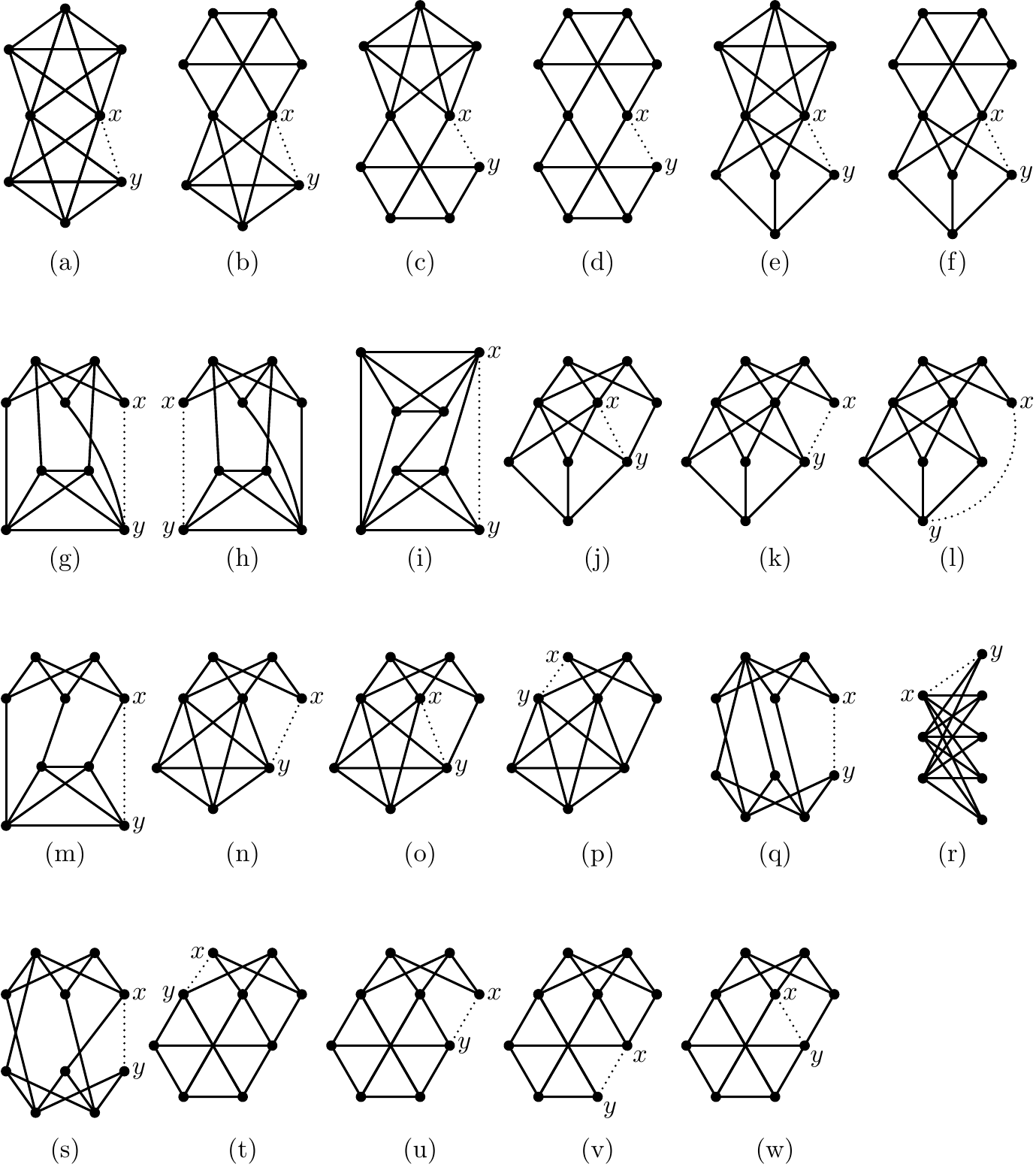}
  \caption{Selected nonplanar graphs in $\Cc_1(\egp)$.}
  \label{fg-selected}
\end{figure}

\begin{lemma}
\label{lm-sep-0}
  $L_x$ and $L_y$ are not\/ $0$-separated.
\end{lemma}

\begin{proof}
  Suppose that $L_x$ and $L_y$ are 0-separated.
  Since $L_x$ is an $x$-K-graph in $G$, there is a path $P$ connecting the boundary of $L_x$ to $L_y$ in $G$.
  Since $P$ does not end on the boundary of $L_y$, $P$ ends at a vertex in the interior of $L_y$.
  Thus $\Bx$ is attached to the interior of $L_y$.
  By symmetry, $\By$ is attached to the interior of $L_x$.
  This contradicts Lemma~\ref{lm-interiors}.
\end{proof}

\begin{figure}
  \centering
  \includegraphics{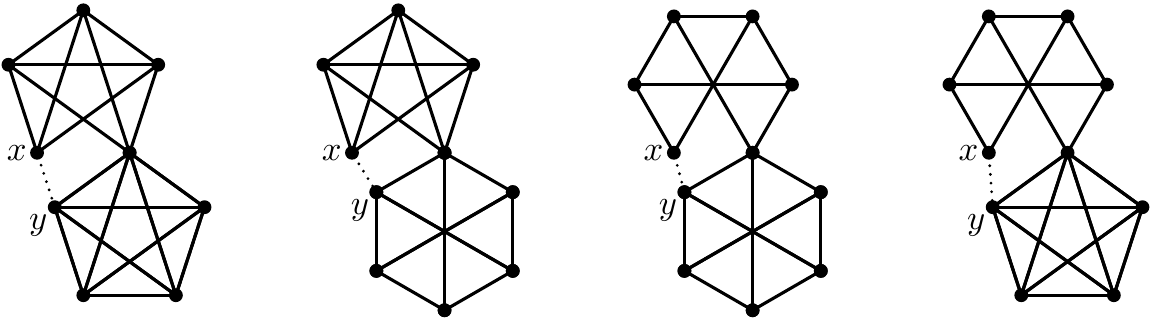}
  \caption{Cascades in $\S_1$ whose $xy$-K-graphs are 1-separated.}
  \label{fg-cascades-1}
\end{figure}

\begin{lemma}
\label{lm-sep-1}
  If $L_x$ and $L_y$ are $1$-separated, then $G$ has one of the graphs in Fig.~\ref{fg-cascades-1} as a minor.
\end{lemma}

\begin{proof}
  We adopt the notation and the assumptions made before Lemma \ref{lm-sep-0}.
  Then we have that $H_x$ admits the linkage \gs{a} or \gs{b}.
  Assume first that $U_y = U = \{u_1\}$.
  Then $H_y$ also admits one of \gs{a} or \gs{b}.
  Let $G_z$ be the $U$-bridge in $G$ containing $L_z$, $z \in \{x, y\}$.
  Since $U$ separates $L_x$ and $L_y$ in $G$, the $U$-bridges $G_x$ and $G_y$ are distinct.
  Since $G$ is nonplanar, one of $G_x$ or $G_y$, say $G_y$ by symmetry, is nonplanar by Theorem~\ref{th-stahl-euler}.
  Suppose that $G_y$ is not isomorphic to a Kuratowski graph. Then there exists a minor-operation $\mu \in \M(G_y)$ such that $\mu G_y$ is nonplanar.
  The graph $\mu G\+$ contains a K-graph and a Kuratowski subgraph whose intersection is either empty or equal to $u_1$.
  Thus, $\egp(\mu G) \ge 2$ by Lemma~\ref{lm-two-k-graphs}(ii), a contradiction with (C1).
  Thus $G_y$ is isomorphic to either $K_5$ or $K_{3,3}$.
  It is not hard to see that $yu_1 \in E(G)$ in both cases.
  We conclude that $G$ has one of the graphs in Fig.~\ref{fg-cascades-1} as a minor.

  Assume now that $U_y = U \cup \{x\}$. In this case $H_y$ is one of \gs{c}, \gs{d}, or \gs{e}.
  Since $L_y$ is linked to $\{u_1, x\}$, there are two choices for the vertices $u_1'$ and $u_2'$.
  In each case, we will be able to find a minor in $G$ isomorphic to one of nonplanar graphs in $\Cc_2(\egp)$ depicted in Figure \ref{fg-selected}. As noted earlier, this contradicts Lemma~\ref{lm-c2-egp}.
  We treat different cases and note that the worst case is always when every foot of the corresponding linkage in Figure \ref{fg-separation} is trivial (i.e. a single vertex), except when this is excluded because that would make $L_x$ and $L_y$ intersect.

  \begin{casesblock}
    \case{$H_y$ is \gs{c}}
    If $u_1' = u_1$ and $u_2' = x$, then $H_y$ contains \gs{d} as a sublinkage (with $u_1'$ being a trivial foot), which is treated in Case 2 below.
    Suppose then that $u_1' = x$ and $u_2' = u_1$.
    If $H_x$ is \gs{a}, then $G$ has~\gsel{a} as a minor.
    If $H_x$ is \gs{b}, then $G$ has~\gsel{b} as a minor.

    \case{$H_y$ is \gs{d}}
    Since \gs{d}\ has a symmetry exchanging its feet, we may assume that $u_1'=u_1$ and $u_2'=x$.
    If $H_x$ is \gs{a}, then $G$ has~\gsel{c} as a minor.
    If $H_x$ is \gs{b}, then $G$ has~\gsel{d} as a minor.

    \case{$H_y$ is \gs{e}}
    If $u_1' = u_1$ and $u_2' = x$, then $H_y$ contains \gs{d} as a sublinkage (having $u_1'$ as a trivial foot), where we contract an edge incident with $y$ and remove the other one.
    Suppose thus that $u_1' = x$ and $u_2' = u_1$.
    If $H_x$ is \gs{a}, then $G$ has~\gsel{e} as a minor.
    If $H_x$ is \gs{b}, then $G$ has~\gsel{f} as a minor.
  \end{casesblock}
\end{proof}

We deal with 2-separated K-graphs similarly.

\begin{lemma}
  \label{lm-2-sep}
  If $L_x$ and $L_y$ are 2-separated, then $G$ has one of the graphs in Fig.~\ref{fg-cascades-2} as a minor.
\end{lemma}

\begin{proof}
  We have that $H_x$ is one of \gs{c}, \gs{d}, or \gs{e}.
  Assume first that $U_y = U = \{u_1, u_2\}$.
  Let $G_z$ be the $U$-bridge containing $L_z$, $z \in \{x, y\}$.
  Since $U$ separates $L_x$ and $L_y$ in $G$, the $U$-bridges $G_x$ and $G_y$ are distinct.
  We will consider $G_x$ and $G_y$ as graphs in $\G_{u_1u_2}$, with terminals $u_1$ and $u_2$.
  Since $G$ is nonplanar, Lemma~\ref{lm-planar-patch} gives that either $G_x^+$ or $G_y^+$ is nonplanar.
  We may assume by symmetry that $G_y^+$ is nonplanar.
  Thus $G_y$ contains a graph in $\Cc_0(\egp)$ as a minor.
  Suppose that there exists a minor-operation $\mu \in \M(G_y)$ such that $\mu G_y^+$ is nonplanar.
  Then $\mu G\+$ contains a K-graph in $G_x$ and a Kuratowski graph that satisfy the conditions of Lemma~\ref{lm-two-k-graphs}(ii) or (iii). (To see this, note that a Kuratowski graph in $\mu G_y^+$ gives rise to a Kuratowski graph in $G^+$ by replacing the edge $u_1u_2$ with a path in $G_x$. The path can be chosen in such a way that it intersects with $L_y$ in a subpath. If this one would not satisfy (ii), then the linkage in $G_x$ is \gs{d} with both feet trivial, and hence we get that (iii) is satisfied.)
  By Lemma~\ref{lm-two-k-graphs}, $\egp(\mu G) \ge 2$, a contradiction.
  We conclude that $G_y$ is isomorphic to one of the three graphs in $\Cc_0(\egp)$ (with terminals $u_1$ and $u_2$).
  Since $L_y$ is 2-linked to $u_1,u_2$, we have that $y \not\in \{u_1, u_2\}$.
  If $H_x$ is \gs{c} or \gs{e}, then $H_x$ contains \gs{d} as a sublinkage.
  Suppose now that $H_x$ is \gs{d}.
  If $G_y$ is isomorphic to \gkur{a}, then $G$ has~\gsel{n} as a minor.
  If $G_y$ is isomorphic to \gkur{b}, then $G$ has~\gsel{u} as a minor.
  If $G_y$ is isomorphic to \gkur{c}, then $G$ has~\gsel{k} or \gsel{l} as a minor.
  In each case, we obtain a contradiction by Lemma \ref{lm-c2-egp}.

  Assume now that $U_y = U \cup \{x\}$.
  Hence $H_y$ is one of \gs{f}, \gs{g}, or \gs{h}.
  \begin{casesblock}
    \case{$H_y$ is \gs{f}}
    This case is symmetric.
    If $H_x$ is \gs{c}, then $G$ has~\gsel{i} as a minor.
    If $H_x$ is \gs{d}, then $G$ has~\gsel{m} as a minor.
    If $H_x$ is \gs{e}, then $G$ has~\gsel{v} as a minor (hint: delete two edges in $H_y$).

    \case{$H_y$ is \gs{g}}
    Suppose that $H_x$ is \gs{c}.
    If $u_2' = u_1$, then $G$ has~\gsel{p} as a minor (hint: contract an edge joining $H_x$ and $H_y$).
    If $u_2' = u_2$, then $G$ has \gsel{t} as a minor (hint: delete two edges in $H_x$). If $u_2' = x$, then $G$ has~\gcas{a} as a minor.

\begin{figure}
  \centering
  \includegraphics{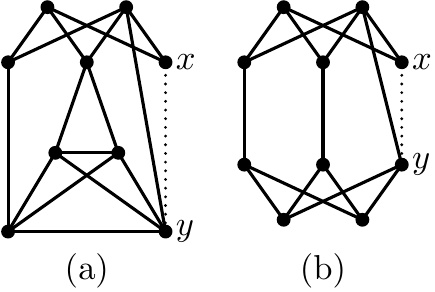}
  \caption{Cascades in $\S_1$ whose $xy$-K-graphs are 2-separated.}
  \label{fg-cascades-2}
\end{figure}

    Suppose now that $H_x$ is \gs{d}.
    If $u_2' = u_1$, then $G$ has~\gsel{t} as a minor.
    If $u_2' = x$, then $G$ has~\gcas{b} as a minor.

    Suppose now that $H_x$ is \gs{e}.
    Since $u_2$-foot is removable in $H_x$ and $u_2'$-foot is removable in $H_y$, Lemma~\ref{lm-removable} asserts that $u_2 \not= u_2'$ (as $L_x$ is a K-graph in $G$).
    If $u_2' = u_1$, then $G$ has~\gsel{v} as a minor (hint: contract one and delete another edge in $H_y$, both incident with the vertex linked to $u_2$).
    If $u_2' = x$, then again, $G$ has \gsel{v} as a minor (hint: delete one and contract the other edge incident with $y$ in $H_y$).

    \case{$H_y$ is \gs{h}}
    Suppose that $H_x$ is \gs{d}.
    If $u_1' = u_1$, then $G$ has~\gsel{w} as a minor (hint: delete one and contract the other edge incident with $y$).
    If $u_1' = x$, then $G$ has~\gsel{r} as a minor.

    Suppose that $H_x$ is \gs{c}.
    If $u_1' = u_1$, then $H_x$ has~\gs{d} as a sublinkage. So this is covered above.
    If $u_1' = u_2$, then $G$ has~\gsel{o} as a minor.
    If $u_1' = x$, then $G$ has~\gsel{r} as a minor.

    Suppose that $H_x$ is \gs{e}.
    Since $u_2$-foot is removable in $H_x$ and $u_2'$-foot and $u_3'$-foot are removable in $H_y$, Lemma~\ref{lm-removable} asserts that $u_1' = u_2$.
    Then $G$ has~\gsel{j} as a minor.
  \end{casesblock}
\end{proof}

For $xy$-K-graphs that are $k$-separated for $k \ge 4$, we shall use the fact that they admit linkages that have many removable feet.

\begin{lemma}
\label{lm-4-linked}
  Suppose that $H$ is a $U$-linkage, where $|U| \ge 4$. Then $H$ has at least $|U|-2$ removable feet.
\end{lemma}

\begin{proof}
  Let $H$ be a $U$-linkage with core $L$, $|U| = k \ge 4$.
  By Lemma~\ref{lm-k4}, $L\cong K_{2,3}$.
  Let $P_1, \ldots, P_k$ be pairwise disjoint paths connecting $L$ and $U = \{u_1, \ldots, u_k\}$
  and suppose that $P_i$ ends at $u_i$, $i=1, \ldots, k$.
  Since $H$ is a $U$-linkage, there are paths $Q_1$ and $Q_2$ connecting the open branches on the boundary
  of $L$ to $U$.
  For $j=1,2$, let $v_j$ be the first vertex on $Q_j$ that belongs to $P_1\cup P_2\cup \cdots\cup P_k$ when traversing $Q_j$ from $L$ towards $U$. Let $i_j$ be the index such that $v_j\in V(P_{i_j})$.
  It is easy to see that, for $i \not= i_1, i_2$, the $u_i$-foot of $H$ is removable.
  Thus $H$ has at least $k-2$ removable feet.
\end{proof}

Let $\B$ be the set of the five $xy$-labeled graphs depicted in Fig.~\ref{fg-bases}.
A graph $H$ is a \df{planar minor} of $G$ if $H$ is a minor of a planar subgraph of $G$.

\begin{figure}
  \centering
  \includegraphics{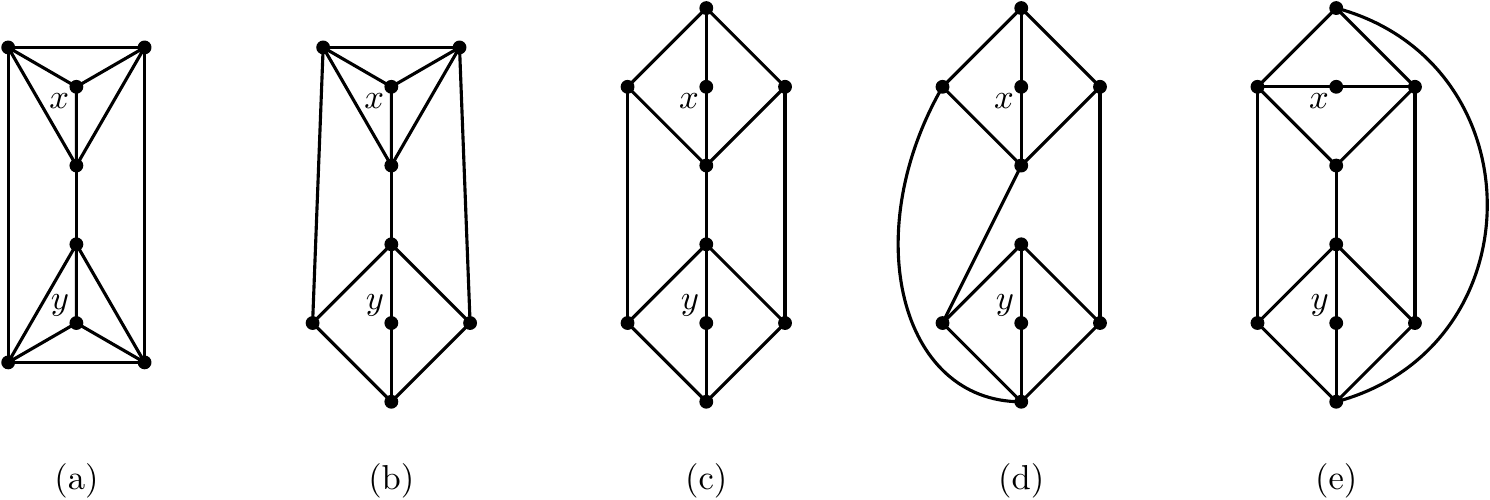}
  \caption{Set $\B$ of bases of cascades in $\S_1$ whose $xy$-K-graphs are $k$-separated for $k \ge 3$.}
  \label{fg-bases}
\end{figure}

\begin{lemma}
\label{lm-3con-bases}
  If $H$ is a base in $G$ such that the $xy$-K-graphs in $H$ are $k$-separated for $k \ge 3$,
  then $G$ contains one of the graphs in $\B$ as a planar minor.
\end{lemma}

\begin{proof}
  We may assume that $H$ does not contain a proper subbase that is $l$-separated for some $l \ge 3$.

  Suppose first that $k=3$.
  We have that $H_x$ is one of \gs{f}, \gs{g}, or \gs{h}.
  Assume first that $U_y = U$.
  In this case, $H_y$ is also one of \gs{f}, \gs{g}, or \gs{h}.
  \begin{casesblock}
    \case{$H_y$ is \gs{h}}
    If $H_x$ is \gs{f}, then $G$ has~\gsel{h} as a minor.
    Suppose that $H_x$ is \gs{g}.
    There are two cases by symmetry:
    If $u_1' = u_1$, then $G$ has~\gsel{t} as a minor.
    (Hint: contract one and delete the other edge incident with $y$ in $H_y$.)
    If $u_1' = u_2$, then $G$ has~\gsel{q} as a minor.

    Suppose now that $H_x$ is \gs{h}.
    There are two cases by symmetry:
    If $u_1' = u_1$, then $G$ has~\gsel{w} as a minor.
    (Hint: Let the two neighbors of $x$ and $y$ be $a,b$ and $c,d$, respectively, where $ac$ and $bd$ is part of the linkage. Then we contract the edges $xb$ and $yc$ and delete the edges $xa$ and $yd$. The vertex $u_1'=u_1$ corresponds to the vertex of degree 4 in \gsel{w}.)
    If $u_1' = u_2$, then $G$ has~\gsel{r} as a minor.
    By symmetry, we may assume now that neither $H_x$ nor $H_y$ is \gs{h}.

    \case{$H_y$ is \gs{f}}
    If $H_x$ is \gs{f}, then $G$ has~\gbase{a} as a planar minor.
    If $H_x$ is \gs{g}, then $G$ has~\gbase{b} as a planar minor.
    By symmetry, we may assume now that neither $H_x$ nor $H_y$ is \gs{f}.

    \case{$H_y$ is \gs{g}}
    The only remaining case is when $H_x$ is \gs{g}.
    If $u_2' = u_2$, then $G$ has~\gbase{c} as a planar minor.
    If $u_2' = u_1$, then $G$ has~\gbase{d} as a planar minor.
  \end{casesblock}

  Assume now that $U_y = U \cup \{x\}$.
  Hence $H_y$ is \gs{i}.
  If $u_2' = x$ or $u_4' = x$, then $H_y$ contains linkage \gs{g} and this case was dealt with above.
  We may thus assume that $u_1' = x$.
  If $H_x$ is \gs{f}, then $G$ has~\gsel{g} as a minor.

  Suppose now that $H_x$ is \gs{g}.
  By Lemma~\ref{lm-removable}, $u_2' \not= u_2$ and $u_4' \not= u_2$.
  Thus $u_3' = u_2$ and $G$ has~\gsel{s} as a minor.
  On the other hand, if $H_x$ is \gs{h}, then Lemma~\ref{lm-removable} gives that $u_2', u_4' \not\in \{u_2, u_3, x\}$ which is impossible.

  Suppose now that $k=4$.
  Assume first that $L_x$ and $L_y$ are 4-separated and suppose that $U_y = U$.
  Thus both $H_x$ and $H_y$ are \gs{i}.
  By Lemma~\ref{lm-removable}, $\{u_2', u_4'\} \cap \{u_2, u_4\} = \emptyset$.
  Thus we may assume by symmetry that $u_1' = u_2$, $u_2' = u_3$, $u_3' = u_4$, and $u_4' = u_1$.
  We conclude that $G$ has~\gbase{e} as a planar minor.

  We may assume now that $U_y = U \cup \{x\}$.
  By Lemma~\ref{lm-4-linked}, $H_y$ has three removable feet.
  Since $H_x$ has two removable feet, there exists $u \in U$ such that the $u$-feet of $H_x$ and $H_y$ are removable.
  By Lemma~\ref{lm-removable}, this contradicts our initial assumption that $H$ does not contain a proper subbase that is $3$-separated.

  Assume now that $k > 4$.
  By Lemma~\ref{lm-4-linked}, there are at most two elements $u$ in $U$ such that either the $u$-foot of $H_x$ or the $u$-foot of $H_y$ is not removable.
  Since $|U| > 4$, there exists $u' \in U$ such that the $u'$-feet of $H_x$ and $H_y$ are removable.
  By Lemma~\ref{lm-removable}, there is a proper subbase of $H$ that is $(k-1)$-separated, a contradiction with our initial assumption about $H$.
\end{proof}

\section{Nonplanar extensions of planar bases}
\label{sc-ext}

Let $\B^*$ be the class of planar graphs that contain a graph in $\B$ as a minor and that are deletion-minimal. These graphs are obtained from $\B$ by splitting vertices of degree 4 in all possible ways such that planarity and minimality are preserved.
It is not hard to check that $\B^*$ contains only five graphs that are not contained in $\B$ (see Fig.~\ref{fg-b-star}).
In this section, we describe the minimal nonplanar graphs that contain a subgraph homeomorphic to a graph in $\B^*$.
Having this description, we use computer to determine the class $\S_1$. The graphs in $\S_1$ that have
a subgraph homeomorphic to a graph in $\B^*$ are depicted in Fig.~\ref{fg-cascades-3}.

\begin{figure}
  \centering
  \includegraphics{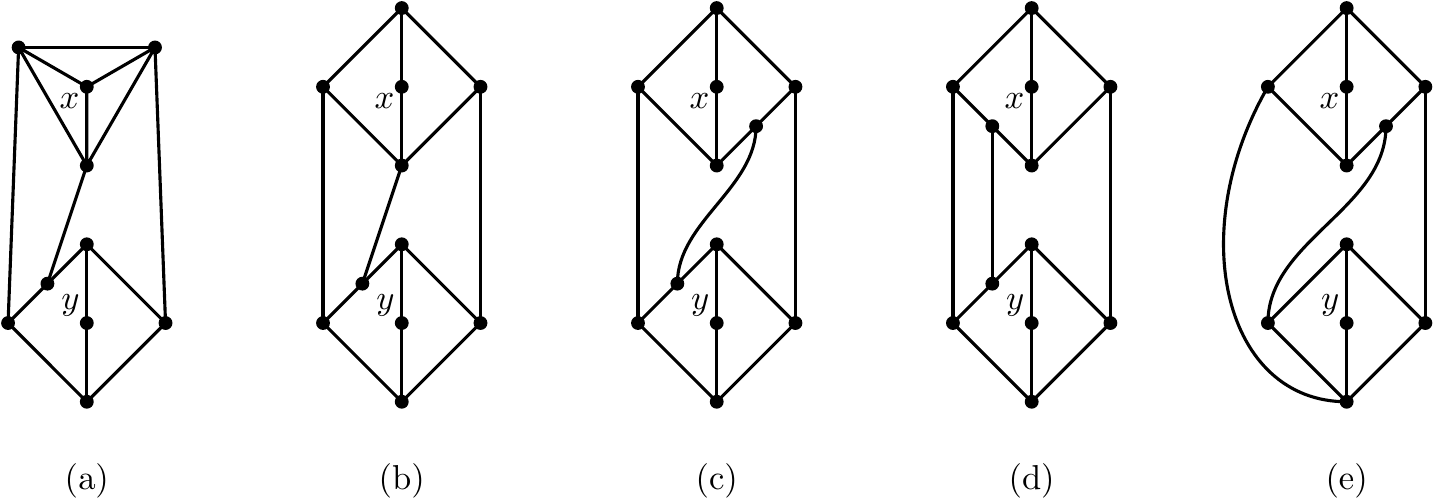}
  \caption{The class $\B^* \sm \B$.}
  \label{fg-b-star}
\end{figure}

Let $H_0$ be a subdivision of $K_{3,3}$, let $v$ be a branch vertex of $H_0$, and let $u_1, u_2, u_3$ be the neighbors of $v$.
The graph $H = H_0 - v$ is called a \df{tripod}. The three (possibly trivial) paths in $H$ with ends $u_1, u_2, u_3$, respectively, are the \df{feet} of $H$.
We say that $H$ is \df{attached} to a subgraph $K$ of $G$ if $H$ is contained in a $K$-bridge $B$, $u_1, u_2, u_3$ are attachments of $B$, and $B$ has no other attachments.
We use the following classical theorem (see~\cite[Theorem~6.3.1]{mohar-book}).

\begin{theorem}
  \label{th-disk-ext}
  Let $G$ be a connected graph and $C$ a cycle in $G$.
  Let $G'$ be a graph obtained from $G$ by adding a new vertex joined to all vertices of $C$.
  Then $G$ can be embedded in the plane with $C$ as an outer cycle unless $G$ contains
  an obstruction of the following type:
  \begin{enumerate}[label=\rm(\alph*)]
  \item
    disjoint paths whose ends are on $C$ and their order on $C$ is interlaced (disjoint crossing paths),
  \item
    a tripod attached to $C$, or
  \item
    a Kuratowski subgraph contained in a 3-connected block of $G'$ distinct from
    the 3-connected block of $G'$ containing $C$.
  \end{enumerate}
\end{theorem}

We formalize homeomorphisms of graphs as follows.
Let $G, H$ be graphs. A mapping $\hem$ with domain $V(H) \cup E(H)$ is called a \df{homeomorphic embedding} of $H$ into $G$
if for every two vertices $v, v'$ and every two edges $e, e'$ of $H$
\begin{enumerate}[label=\rm(\roman*)]
\item
  $\hem(v)$ is a vertex of $G$, and if $v,v'$ are distinct then $\hem(v),\hem(v')$ are distinct,
\item
  if $e$ has ends $v, v'$, then $\hem(e)$ is a path in $G$ with ends $\hem(v), \hem(v')$,
  and otherwise disjoint from $\hem(V(H))$, and
\item
  if $e,e'$ are distinct, then $\hem(e)$ and $\hem(e')$ are edge-disjoint,
  and if they have a vertex in common, then this vertex is an end of both.
\end{enumerate}
We shall denote the fact that $\hem$ is a homeomorphic embedding of $H$ into $G$ by writing $\hem: H \homeo G$.
If $K$ is a subgraph of $H$, then we denote by $\hem(K)$ the subgraph of $G$ consisting of all vertices $\hem(v)$,
where $v \in V(H)$, and all vertices and edges that belong to $\hem(e)$ for some $e \in E(K)$.
Note that $\hem(V(K))\subseteq V(\hem(K))$ mean different sets. It is easy to see that $G$
has a subgraph homeomorphic to $H$ if and only if there is a homeomorphic embedding $H \homeo G$.
An \df{$\hem$-bridge} is an $\hem(H)$-bridge in $G$; an \df{$\hem$-branch} is an image of an edge of $H$.
A bridge is \df{local} if all its vertices of attachment are on a single branch $\hem(e)$, $e\in E(H)$.

The following result is well-known (see~\cite{mohar-book}, Lemma 6.2.1).

\begin{lemma}
\label{lm-no-local-bridges-3con}
Let $H$ be a graph with at least three vertices and $\hem$ a homeomorphic embedding of $H$ into a 3-connected graph $G$.
Then there exists a homeomorphic embedding $\hem'$ such that:
\begin{enumerate}[label=\rm(\roman*)]
\item
  $\hem(v) = \hem'(v)$ for each $v \in V(H)$.
\item
  $\hem'(e)$ is a path that is contained in the union of $\hem(e)$ and all local $\hem(e)$-bridges.
\item
  There are no local $\hem'$-bridges.
\end{enumerate}
\end{lemma}

In order to apply Lemma~\ref{lm-no-local-bridges-3con} to a base in $\B^*$, we need to assure that
new homeomorphic embedding still maps terminals to terminals. We will need the following lemmas.

\begin{lemma}
  \label{lm-cascade-kur}
  Suppose that $G \in \S_1$ has a base homeomorphic to a graph in $\B^*$ and that $K$ is
  a Kuratowski subgraph of $G$. If none of the branch vertices of $K$ lie in $L_x$, then two
  of its open branches intersect $L_x$. The same holds for the intersection of $K$ with $L_y$.
\end{lemma}

\begin{proof}
  Assume for a contradiction that $K$ is disjoint from $L_x$
  except possibly for an open branch $P$ of $K$.
  By inspection of graphs in $\B^*$, we see that there is an edge $e$ incident with $L_x$ such that
  $L_x$ is an $x$-K-graph in $G/e$ and there is a Kuratowski subgraph $K'$ in $G /e$
  that shares at most one half-open branch with $L_x$.
  By Lemma~\ref{lm-two-k-graphs}, $\egp(G/e) \ge 2$.
  Since $G/e$ is nonplanar, this contradicts the condition (C1) from the definition of cascades.
\end{proof}

\begin{lemma}
  \label{lm-3-connected}
  Let $U$ be a vertex-cut in $G \in \S_1$. If\/ $|U| \le 2$, then each nontrivial $U$-bridge in $G$ contains either $x$ or $y$.
\end{lemma}

\begin{proof}
  Let $B$ be a nontrivial $U$-bridge that contains neither $x$ nor $y$.
  If $|U| = 1$, let $G_1 = G - B^\circ$.
  If $U = \{u, v\}$ has size 2, let $G_1 = G - B^\circ + uv$.
  Since Kuratowski graphs are 3-connected, $G_1$ contains the same disjoint $xy$-K-graphs as $G$.
  Thus $\egp(G_1) = 2$ by Lemma~\ref{lm-two-k-graphs}(i). Since $G_1$ is a proper minor of $G$, $\eg(G_1) = 0$ by (C1).
  If $|U| = 1$, Theorem~\ref{th-stahl-euler} implies that $B$ is nonplanar since $G$ is nonplanar.
  If $|U| = 2$, then Lemma~\ref{lm-planar-patch} implies that $B +uv$ is nonplanar since $G$ is nonplanar.
  We may assume by symmetry that $|V(L_y) \cap U| \le 1$.
  Let us now consider an edge $e \in E(L_x)$ (with $e\ne uv$ if $|U| = 2$) and the graph $G_0 = G / e$.
  The graph $G_0$ is nonplanar since it contains $B$ or $B+uv$ as a minor.
  Also $G_0^+$ contains a Kuratowski subgraph in $B$ and a K-graph $L_y$ that intersect in at most one vertex
  or in at most one half-open branch.
  Lemma~\ref{lm-two-k-graphs}(ii) gives that $\egp(G_0) = 2$. This is a contradiction with (C1).
\end{proof}

\begin{lemma}
  \label{lm-k23-local}
  Let $H$ be a base of a graph\/ $G \in \S_1$, and $\hem: H \homeo G$ a homeomorphic embedding of $H$ in $G$.
  If $L_x\cong K_{2,3}$ and $P$ is the branch of $L_x$
  that contains the interior of $L_x$, then
  there are no local $\hem$-bridges with attachments only on $P$.
\end{lemma}

\begin{proof}
  Let $C$ be the boundary of $L_x$ which consists of the $\hem$-branches $P_1, P_2$.
  Assume first that there is an $\hem$-bridge $B_0$ that attaches only at the ends $w_1, w_2$ of $P$.
  By Lemma~\ref{lm-3-connected}, $B_0$ is trivial and consists of the edge $w_1w_2$.
  Let $G' = G - w_1w_2$. Since $\egp(G') \ge 2$, we have that $G'$ is planar by (C1).
  Since $G$ is nonplanar, there are paths $P_3$ and $P_4$ connecting $P - w_1 - w_2$
  to $P_1 - w_1 - w_2$ and $P_2 - w_1 - w_2$, respectively.
  Let $P_5$ be a path in $\By$ connecting $P_1 - w_1 - w_2$ and $P_2 - w_1 - w_2$.
  The planarity of $G'$ implies that $P_3$ and $P_4$ are internally disjoint from $L_y$, and therefore
  $L_x \cup P_3 \cup P_4 \cup P_5 \cup w_1w_2$ contains a Kuratowski subgraph $K$.
  The intersection of $K$ with $L_y$ is contained in $P_5$.
  This contradicts Lemma~\ref{lm-cascade-kur}.

  We may assume now that all $\hem$-bridges that are attach to $P$, have a vertex of attachment in the interior of $P$.
  Let $B'$ be a local $\hem$-bridge with an attachment $t\in V(P)\setminus \{w_1,w_2\}$.
  Let $G'$ be the graph obtained from $G$ by deleting an edge $e$ of $B'$ incident with $t$.
  Since $\egp(G') \ge 2$, we have that $G'$ is planar by (C1). Let $B$ be the $C$-bridge containing $\cup B'$.
  Since $G$ is nonplanar, $B$ cannot be drawn inside
  a disk with $C$ on the boundary. By Theorem~\ref{th-disk-ext}, there are three possibilities.
  The option (iii) contradicts Lemma~\ref{lm-cascade-kur}.
  Suppose that (i) holds and let $P_3, P_4$ be a pair of crossing paths.
  Since $B$ is connected, there is a path $P_5$ connecting interiors of $P_3$ and $P_4$.
  Thus $C \cup P_3 \cup P_4 \cup P_5$ is a $K_{3,3}$-minor which contradicts Lemma~\ref{lm-cascade-kur}.
  Suppose now that (ii) holds and there is a tripod $T$ in $C\cup B$.
  If $T$ has a foot of nonzero length, then $C \cup T$ contains a $K_{3,3}$-minor.
  Otherwise, there is a path $P_5$ connecting the two triads that $T$ consists of.
  Hence $C \cup T \cup P_5$ contains a $K_5$-minor.
  In both cases, Lemma~\ref{lm-cascade-kur} yields a contradiction.
\end{proof}

Let $H$ be a planar 3-connected graph and $\hem$ a homeomorphic embedding of $H$ into $G$.
A well-known result of Tutte~\cite{tutte-1963} says
that $\hem(H)$ has a unique embedding in the plane where each face is a cycle.
Let us call each such a cycle an \df{$\hem$-face}.
An \df{$\hem$-path} is a path in $G$ with ends in $\hem(H)$ but otherwise disjoint from $\hem(H)$.
An \df{$\hem$-jump} is an $\hem$-path such that no $\hem$-face includes both of its ends.

An \df{$\hem$-cross} consists of two disjoint $\hem$-paths $P_1, P_2$ with ends $u_1, v_1$ and $u_2, v_2$ (respectively) on a common $\hem$-face such that
the ends appear in the interlaced order $u_1, u_2, v_1, v_2$ on the boundary of the face.
An $\hem$-cross $P_1, P_2$ is \df{free} if neither $P_1$ nor $P_2$ has its ends on $\hem(e)$ for a single $e \in E(H)$
and, whenever the ends of $P_1$ and $P_2$ are in $V(\hem(e_1)) \cup V(\hem(e_2))$ for $e_1, e_2 \in E(H)$, then
$e_1$ and $e_2$ have no end in common.

An \df{$\hem$-triad} is an $\hem(H)$-bridge $B$ with three attachments that consists of three internally disjoint paths $P_1, P_2, P_3$
connecting the attachments to a vertex $v \in V(G) \sm V(\hem(H))$.
Furthermore, every pair of attachments of $B$ lie on a common $\hem$-face but no $\hem$-face contains all the attachments.

An \df{$\hem$-tripod} in $G$ is a tripod whose feet are in $\hem(H)$, but none of its other vertices or edges is in $\hem(H)$.
Let $C$ be an $\hem$-face and $v_1, v_2, v_3 \in V(C)$ branch-vertices of $\hem(H)$.
Let $Q$ be the union of one or two $\hem$-branches, each with both ends in $\{v_1, v_2, v_3\}$.
A \df{weak $\hem$-tripod} is a tripod $B$ in $G$ with attachments $v_1,v_2,v_3$ such that $B\cap \hem(H) = Q\cup \{v_1,v_2,v_3\}$ (see Figure \ref{fg-weak-tripod}, where $v_1,v_2,v_3$ correspond to the square vertices).

We will use the following well-known result.

\begin{lemma}
  \label{lm-faces}
  Let $G$ be a subdivision of a 3-connected plane graph.
  Then each pair of intersecting faces of $G$ share either a single branch-vertex or a single branch.
\end{lemma}

We say that a graph $G \in \Gcxy$ is \df{essentially 3-connected} if $G\+$ is 3-connected.
The following lemma and its proof are adapted from~\cite{robertson-2001}.

\begin{lemma}
  \label{lm-ext}
  Suppose that\/ $G \in \S_1$ has a base homeomorphic to a graph $H \in \B^*$.
  Then there exists a homeomorphic embedding $\hem: H \homeo G$, mapping the terminals of $H$ to the terminals of $G$, such that one of the following holds:
  \begin{enumerate}[label=\rm(W\arabic*)]
  \setlength{\itemindent}{4mm}
  \item
    There exists an $\hem$-jump.
  \item
    There exists a free $\hem$-cross.
  \item
    There exists an $\hem$-tripod or a weak $\hem$-tripod.
  \item
    There exists an $\hem$-triad.
  \end{enumerate}
\end{lemma}

\begin{proof}
  Since $H$ has three internally disjoint paths joining the two terminals and $\hem$ maps the terminals of $H$ to $x$ and $y$,
  Lemma~\ref{lm-3-connected} gives that $G$ is essentially 3-connected.
  By Lemmas~\ref{lm-no-local-bridges-3con} and~\ref{lm-k23-local}, there exists a homeomorphic embedding $\hem$ from $H$ into $G$ such that there are no local $\hem$-bridges
  and terminals are mapped onto terminals by $\hem$.
  Suppose that none of (W1)--(W4) holds for $\hem$.
  Let $B$ be an $\hem$-bridge and $S$ the set of attachments of $B$.
  By excluding (W1), any two elements of $S$ lie on the same $\hem$-face.
  Not having (W4), each triple in $S$ must lie on the same $\hem$-face.
  We claim that all vertices in $S$ are contained in one of the faces.
  To see this, we will use induction.
  Let $k \ge 3$ and let us assume that for each subset $S'$ of $S$ of size $k$, there
  exists an $\hem$-face $F$ such that $S'$ lie on $F$.
  We shall prove that the same holds for each subset of $S$ of size $k+1$.
  Suppose for a contradiction that $S_0 = \{v_1, \ldots, v_{k+1}\}$ is a subset of $S$ of size $k+1$ such that there is no $\hem$-face
  that contains $S_0$. For $i = 1, \ldots, k+1$, let $F_i$ be the $\hem$-face that contains $S_0 \sm \{v_i\}$.
  Thus $F_i$ are pairwise distinct. In particular, each vertex $v_i$ belongs to $k\ge3$ distinct faces in $\{F_1,\dots,F_{k+1}\} \setminus \{F_i\}$ and thus $v_i$ is a branch vertex of $\hem(H)$.
  Since $v_1$ and $v_2$ belong to both $F_3$ and $F_4$, Lemma~\ref{lm-faces} gives that there is an $\hem$-branch $P_{12}$ that contains $v_1$ and $v_2$.
  Similarly, there is an $\hem$-branch $P_{ij}$ for each pair $i,j = 1,\ldots, k+1$.
  The branch vertices $v_i$ and the paths $P_{ij}$ form a subdivision of $K_{k+1}$.
  This implies that $k = 3$. However, for graphs in $\B^*$, no subgraph isomorphic to $K_4$ has each triple of its vertices on the same face. With this contradiction we conclude that there exists an $\hem$-face that contains $S$.

\begin{figure}
  \centering
  \includegraphics{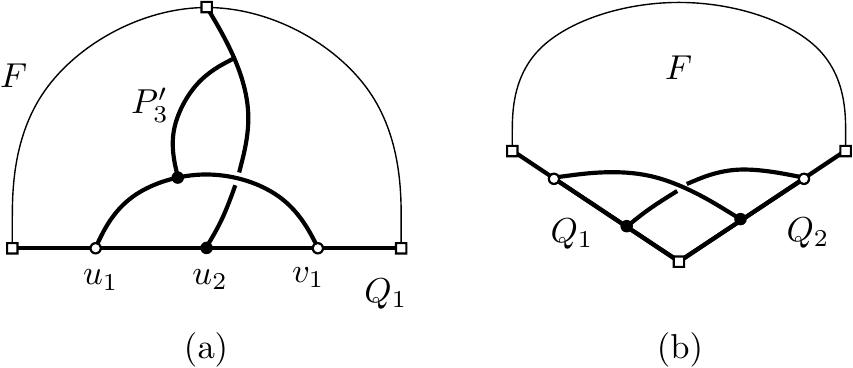}
  \caption{Weak tripods from the proof of Lemma \ref{lm-ext}.}
  \label{fg-weak-tripod}
\end{figure}

  Since there are no local $\hem$-bridges, for each $\hem$-bridge $B$, there exists a unique $\hem$-face $F_B$ such that $F_B$ contains all attachments of $B$.
  For an $\hem$-face $F$, let $G_F$ be the union of all $\hem$-bridges whose attachments are contained in $F$.
  Since $G$ is nonplanar, there exists an $\hem$-face $F$ such that $G_F$ does not embed inside $F$.
  By excluding (W3) and by Theorem~\ref{th-disk-ext}, there is an $\hem$-cross $P_1, P_2$ in $F$.
  Let $u_i, v_i$ be the ends of $P_i$, $i = 1,2$.
  Pick $P_1$ and $P_2$ so that number of pairs in $\{u_1, v_1, u_2, v_2\}$ that lie on a single $\hem$-branch is minimized.
  Assume first that $u_1$ and $v_1$ lie on a single $\hem$-branch $Q_1$.
  Since the bridge containing $P_1$ is not local, there is a path $P_3$ connecting $P_1$ and
  an $\hem$-branch $Q_2$ distinct from $Q_1$.
  If also $u_2$ and $v_2$ lie on $Q_1$, then this yields a contradiction as $P_1 \cup P_2 \cup P_3$
  contains an $\hem$-cross where the ends do not lie on a single $\hem$-branch.
  Thus we may assume that the pair $u_2, v_2$ does not share a common $\hem$-branch.
  If $P_3$ is disjoint from $P_2$, then $P_1 \cup P_2 \cup P_3$ contains an $\hem$-cross where the pairs $u_1, v_1$ and $u_2, v_2$
  do not share a common $\hem$-branch. This again contradicts the choice of $P_1$ and $P_2$.
  If $P_3$ intersects $P_2$ (even if only at its endpoint), let $P_3'$ be the subpath of $P_3$ from $P_1$ to the first vertex on $P_2$, and let $P$ be the path in $Q_1$ from $u_1$ to $v_1$.
  Then $P \cup P_1 \cup P_2 \cup P_3$ forms a weak $\hem$-tripod (see Figure \ref{fg-weak-tripod}(a)).
  This gives (W3).
  Finally, we may assume by symmetry that none of the pairs $u_1, v_1$ and $u_2, v_2$ share a common $\hem$-branch.
  Then we have (W2), unless there are two $\hem$-branches $Q_1, Q_2$ that share a branch vertex and
  so that $u_1, u_2$ lie on $Q_1$ and $v_1, v_2$ lie on $Q_2$.
  This gives (W3) as $P_1 \cup P_2 \cup Q_1 \cup Q_2$ contains a weak $\hem$-tripod (see Figure \ref{fg-weak-tripod}(b)).
  We conclude that $\hem$ satisfies one of (W1)--(W4).
\end{proof}

Even a stronger version of Lemma~\ref{lm-ext} can be proved.

\begin{lemma}
\label{lm-ext-strong}
  Let $G \in \S_1$ that has a base homeomorphic to a graph $H \in \B^*$.
  Then there exists a homeomorphic embedding $\hem: H \homeo G$ such that one of the following holds:
  \begin{enumerate}[label=\rm(T\arabic*)]
  \setlength{\itemindent}{4mm}
  \item
    There exists an $\hem$-jump.
  \item
    There exists an $\hem$-cross that attaches onto branch-vertices of $\hem(H)$.
  \item
    There exists a (weak) $\hem$-tripod with trivial feet that attaches onto branch-vertices of $\hem(H)$.
  \item
    There exist branch-vertices $u_1, u_2, u_3$ of $\hem(H)$ such that no two of them lie on a common $\hem$-branch and there exists an $\hem$-triad that attaches onto $u_1$, $u_2$, and $u_3$.
  \end{enumerate}
  Moreover, $G$ is the union of $\hem(H)$ and the corresponding obstruction in {\rm (T1)--(T4)}.
\end{lemma}

\begin{proof}
  Lemma~\ref{lm-ext} yields a homeomorphic embedding $\hem: H \homeo G$ such that
  one of (W1)--(W4) holds.
  Let $\mu \in \M(G)$. If $\mu G$ admits a homeomorphic embedding $\hem' : H \homeo \mu G$
  that satisfies one of (W1)--(W4), then $\egp(\mu G) \ge \egp(H) = 2$ and $\mu G$ is nonplanar.
  This contradicts the property (C1) of cascades.
  Let us describe sufficient conditions that yield this contradiction.
  Clearly, if $\mu$ is deletion of an edge $e \not\in \hem(H)$ that also does not appear in the obstruction given by (W1)--(W4),
  then $\hem, G - e$ contradicts (C1). This yields the last statement in the lemma.
  If $\mu$ is a contraction of an edge $e \in \hem(H)$ and one of its ends is not a terminal or a branch-vertex of $\hem(H)$
  and, furthermore, the ends of $e$ are not attachments of the obstruction given by (W1)--(W4), then there is a homeomorphic
  embedding $\hem': H \homeo \mu G$ that satisfies one of (W1)--(W4), a contradiction.

  Suppose that none of (T1)--(T4) holds. Thus one of (W2)--(W4) holds.
  Assume first that (W2) or (W3) holds and let $B$ be the union of $\hem$-bridges as given by (W2) or (W3).
  Let $C$ be an $\hem$-face of $\hem(H)$ that contains all attachments of $B$.
  Let us prove that $C$ contains no terminals.
  Suppose to the contrary that $C$ contains $x$.
  Thus $C \cup B$ contains a K-graph of $G$.
  Let $e \in E(\hem(H)) \sm E(L_x)$ be an edge that is incident with $L_x$ and not incident with $C$.
  By inspection of $\B^*$, the graph $G/e$ contains two disjoint K-graphs and $C \cup B$ contains
  a K-graph of $G/e$. This contradicts (C1) by Lemma~\ref{lm-two-k-graphs}(i).
  Hence we may assume that $C$ contains no terminals.

  Assume that (W2) holds.
  Since (T2) does not hold, there is a free $\hem$-cross $P_1, P_2$
  such that $P_1$ has attachment $u_1$ on an open branch $Q$ of $\hem(H)$.
  Let $u_1, v_1$ and $u_2, v_2$ be the attachments of $P_1$ and $P_2$, respectively.
  Let $e_1$ and $e_2$ be the edges of $Q$ incident with $u_1$.
  Consider the graphs $G_1 = G /e_1$ and $G_2 = G/e_2$.
  Since $Q$ is an $\hem$-branch of length at least 2, $\hem$ induces homeomorphic embeddings $\hem_1: H \homeo G_1$ and
  $\hem_2: H \homeo G_2$.
  Suppose $P_1, P_2$ is a free $\hem_1$-cross. Then $G_1$ is nonplanar and, since $G_1$ has a base $\hem_1(H)$, we have that $\egp(G_1) \ge 2$.
  This contradicts (C1). Thus $P_1, P_2$ is not a free $\hem_1$-cross.
  Similarly $P_1, P_2$ is not a free $\hem_2$-cross.
  Let $e_1 = u_1w_1$ and $e_2 = u_1w_2$.
  Since $P_1, P_2$ is a free $\hem$-cross, we may assume that $w_1 \not\in \{u_2, v_2\}$.
  Since $P_1, P_2$ is not a free $\hem_1$-cross, $P_1$ has both ends on a single $\hem_1$-branch $Q_1$.
  If $w_2 \in \{u_2, v_2\}$, then the ends of $P_1, P_2$ lie on $Q \cup Q_1$ in $G$, a contradiction.
  Thus we may assume that $w_2 \not\in \{u_2, v_2\}$ and we obtain by symmetry that $P_1$ has both ends on a single $\hem_2$-branch $Q_2$.
  We conclude that $Q, Q_1, Q_2$ is a subdivision of a triangle, $v_1$ is the common vertex of $Q_1$ and $Q_2$, $w_1$ is the common vertex of $Q$ and $Q_1$,
  $w_2$ is the common vertex of $Q$ and $Q_2$, and $u_2, v_2$ lie on $Q_1, Q_2$, respectively.
  Let $e_3$ be the edge of $Q_1$ that is incident with $u_2$ and lies between $w_1$ and $u_2$.
  Consider the graph $G_3 = G/e_3$. Clearly, $G_3$ satisfies either (W2) or (W3) and thus contradicts (C1).

  Assume that (W3) holds.
  Since (T3) does not hold, either $B$ has a nontrivial foot or $B$
  has an attachment on an open $\hem$-branch.
  Suppose first that $B$ has a nontrivial foot $P$.
  Since $P$ contains no terminals, contracting $P$ preserves (W3).
  This is a contradiction with the observation made above.
  Suppose now that $u$ is an attachment of $B$ on an open $\hem$-branch $Q_1$.
  Let $uv \in E(Q_1)$ be an edge incident with $u$.
  Since $u$ is not a terminal and $v$ is not an attachment of $B$ by Observation \ref{obs-triangle}, $G /uv$ contradicts (C1).

  Assume that (W4) holds.
  By Observation \ref{obs-triangle}, the attachments $u_1, u_2, u_3$ of $B$ are independent.
  Since (T4) does not hold, we may assume that $u_1$ lies on an open $\hem$-branch $Q$.
  If $u_1 = x$ and $L$ is the $x$-K-graph in $\hem(H)$, then $L \cup B$ contains a Kuratowski subgraph
  that is disjoint from the $y$-K-graph, a contradiction with Lemma~\ref{lm-cascade-kur}.
  Thus we may assume that $u_1$ is not a terminal.
  Let $u_1w_1, u_1w_2$ be the edges of $Q$ incident with $u_1$ and
  let $G_1 = G/u_1w_1$ and $G_2/u_1w_2$.
  Since both $G_1$ and $G_2$ admit a homeomorphic embedding of $H$, they are both planar.
  Thus there is an $\hem$-face that contains the vertices $w_1, u_2, u_3$ and an $\hem$-face
  that contains the vertices $w_2, u_2, u_3$. It is not hard to see that $u_2, u_3$ is a 2-vertex-cut in $G$ that blocks $L_x$ and $L_y$,
  a contradiction.
\end{proof}

The list of minimal graphs satisfying the conditions of Lemma~\ref{lm-ext-strong} was generated by computer\footnote{The programs used and the graphs generated are archived at \url{arXiv.org} along with the original manuscript of this paper.} and checked for which of them are in $\S_1$.
The outcome of this computation is the following theorem.
A proof by hand would be possible but would involve detailed case analysis that can be as error-prone
as a computer program.

\begin{figure}
  \centering
  \includegraphics{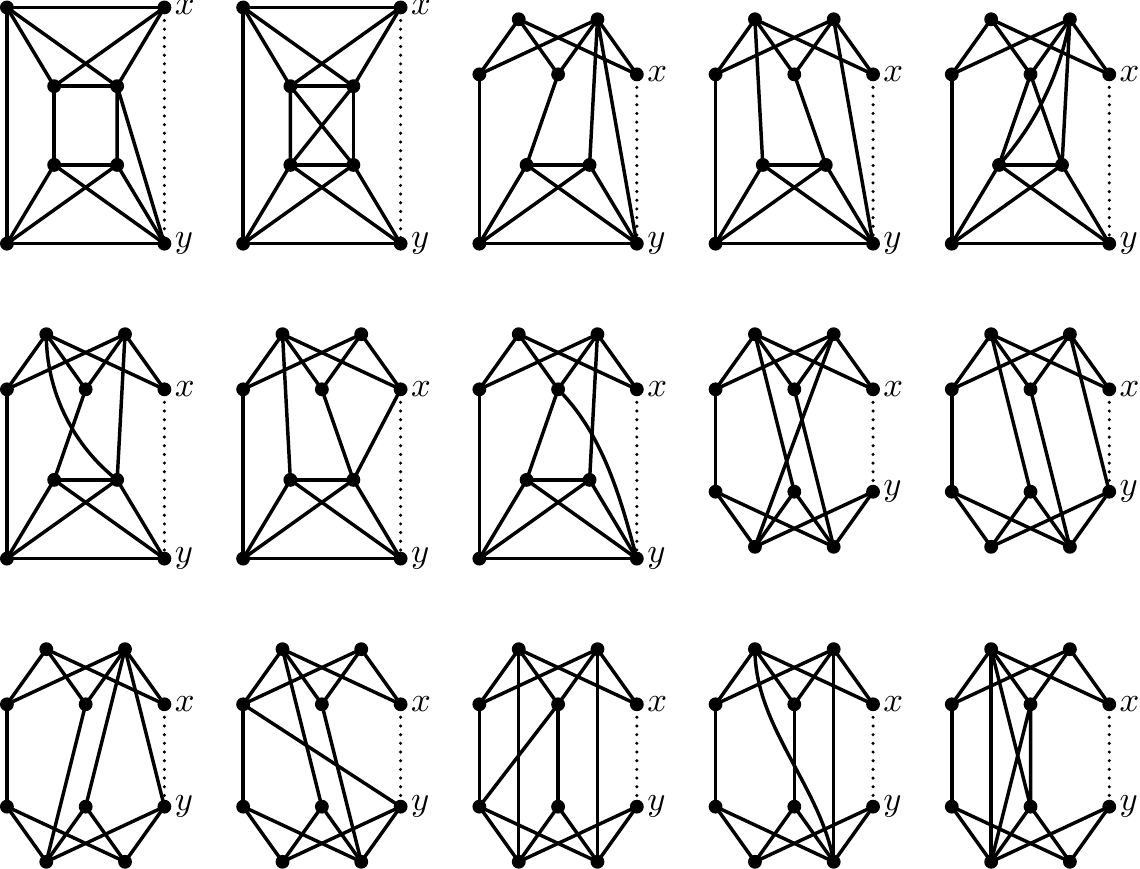}
  \caption{Cascades in $\S_1$ whose $xy$-K-graphs are $k$-separated for $k \ge 3$.}
  \label{fg-cascades-3}
\end{figure}

\begin{theorem}
  The class $\S_1$ consists of $21$ graphs which are depicted in Figs.~\ref{fg-cascades-1}, \ref{fg-cascades-2}, and~\ref{fg-cascades-3}.
\end{theorem}

\begin{proof}
  Let us give detailed overview of the proof and indicate which parts of the proof rely on computer verification.
  Let $\C$ be the set of $21$ graphs depicted in Figures \ref{fg-cascades-1}, \ref{fg-cascades-2}, and~\ref{fg-cascades-3}.
  To show that $\C \ss \S_1$, we have to prove that each graph $G \in \C$ satisfies (C1)--(C3) and $\egp(G) = 2$.
  We are not aware of a faster method than computing $\eg(\mu G)$ and $\egp(\mu G)$ for all minor-operations $\mu \in M(G)$
  and then checking that (C1)--(C3) were satisfied. This was verified by computer for every graph in $\C$.

  In order to show that $\S_1 \ss \C$, let us consider a graph $G \in \S_1$.
  By Lemma~\ref{lm-disjoint-xy-k-graphs}, $G$ contains disjoint $xy$-K-graphs that are $k$-separated for some $k \ge 0$.
  If $k \le 2$, then Lemmas~\ref{lm-sep-0},~\ref{lm-sep-1}, and~\ref{lm-2-sep} give that $G \in \C$.
  If $k \ge 3$, then Lemma~\ref{lm-3con-bases} asserts that $G$ has a base that is homeomorphic to a graph $H \in \B^*$.
  By Lemma~\ref{lm-ext-strong}, there is a homeomorphic embedding of $H$ into $G$ such that one of (T1)--(T4) holds.
  By computer, we have constructed all those graphs (which yields several hundred) and
  verified that all of these graphs that satisfy (C1)--(C3) belong to~$\C$.
\end{proof}

\bibliographystyle{abbrv}
\bibliography{bibliography}

\end{document}